\documentclass[11pt]{article}
\usepackage{amssymb, amsmath, amsthm, xstring, mathrsfs, mathtools, float}

\newcommand{\spa}[1]{
	\IfEqCase{#1}{ 
		{1}{\:}
		{2}{\:\:}
		{3}{\:\:\:}
		{4}{\:\:\:\:}
		{5}{\:\:\:\:\:}
		{6}{\:\:\:\:\:\:}
		{7}{\:\:\:\:\:\:\:}
		{8}{\:\:\:\:\:\:\:\:}
		{9}{\:\:\:\:\:\:\:\:\:}
		{10}{\:\:\:\:\:\:\:\:\:\:}
		{11}{\:\:\:\:\:\:\:\:\:\:\:}
		{12}{\:\:\:\:\:\:\:\:\:\:\:\:}
		{13}{\:\:\:\:\:\:\:\:\:\:\:\:\:}
		{14}{\:\:\:\:\:\:\:\:\:\:\:\:\:\:}
		{15}{\:\:\:\:\:\:\:\:\:\:\:\:\:\:\:}
		{16}{\:\:\:\:\:\:\:\:\:\:\:\:\:\:\:\:}
		{17}{\:\:\:\:\:\:\:\:\:\:\:\:\:\:\:\:\:}
		{18}{\:\:\:\:\:\:\:\:\:\:\:\:\:\:\:\:\:\:}
		{19}{\:\:\:\:\:\:\:\:\:\:\:\:\:\:\:\:\:\:\:}
		{20}{\:\:\:\:\:\:\:\:\:\:\:\:\:\:\:\:\:\:\:\:}
		{-1}{\!}
		{-2}{\!\!}
		{-3}{\!\!\!}
		{-4}{\!\!\!\!}
		{-5}{\!\!\!\!\!}
	}[\PackageError{spa}{Undefined option to spa: #1}{}]
}

\newcommand{\subST}{_{_{ST}}}

\newcommand{\kaufST}[1]{\langle #1 \rangle\subST}
\newcommand{\kauf}[1]{\langle #1 \rangle}

\usepackage{graphicx}
\usepackage{enumitem}

\theoremstyle{plain}
\newtheorem{thm}{Theorem}
\newtheorem{lmm}{Lemma}

\newtheorem{corol}{Corollary}

\newtheorem{prop}{Proposition}

\newtheorem{mth}{Method}

\theoremstyle{definition}

\newtheorem{exmp}{Example}
\theoremstyle{remark}

\date{}
\title{\textbf{On wrapping number, adequacy and the crossing number of satellite knots}}
\date{} 
\author{Adri\'an Jim\'enez Pascual}

\begin{document}
\maketitle

\begin{abstract}
In this work we establish the tightest lower bound up-to-date for the minimal crossing number of a satellite knot based on the minimal crossing number of the companion used to build the satellite. If $M$ is the wrapping number of the pattern knot, we essentially show that $c(Sat(P,C))>\frac{M^2}{2}c(C)$. The existence of this bound will be proven when the companion knot is adequate, and it will be further tuned in the case of the companion being alternating.
\end{abstract}

\section{Introduction}
This work is divided into two parts: the first one explores the concept of \emph{wrapping number}, presenting two methods that will explicitly extract the wrapping number of any given knot diagram in the annulus; the second part focuses on \emph{link adequacy}, proving some known and some new results regarding adequate links and their parallel versions. These results on adequacy will then help to establish the basis for proving results concerning the \emph{crossing number} of link parallels and, later, satellite knots. In this direction, we will prove first (\emph{Theorem \ref{thmCrossParal}}) that
\[c(L^r)\geq \frac{r^2}{2}c(L)+2r-1,\]
where $L^r$ is the $r$-parallel of the adequate oriented link $L$. We will then, before directly working with satellite knots, bring to specific terms the concept of \emph{grafting composition}. The concept of grafting is not new to knot theory: it is the generalization of the connected sum of knots, where instead of ``cutting and pasting'' one strand of each knot we use several. Nonetheless, we will take the chance to define it rigorously so that its usage is well settled. Then, we will proceed to establish some results regarding the breadth of knots and their parallels and satellites, among which the main lemma will be the following (\emph{Lemma} \ref{lemmamain}):
\[B(J(K;r))\geq B(J(K^r)).\]
The terms in this inequality represent, respectively, the breadth of the $r$-cable and $r$-parallel of the oriented adequate knot $K$. Finally, we will conclude with the main theorem of this work (\emph{Theorem \ref{thmMain}}), which brings together all previous concepts and results and reads:
\[c(Sat(P,C))\geq(M_{\beta_M}-m_{\beta_M})+\frac{M^2}{2}c(C)+2M-1,\]
assuming $C$ is adequate, and being the Jones polynomial of the pattern $P$ in the solid torus $J\subST(P)=\sum_{k=0}^M\beta_kz\subST^k$, and $M_{\beta_M}$ and $m_{\beta_M}$ respectively the maximum and minimum exponents of $t$ in $\beta_M\in\mathbb{Z}[t^{-1/2},t^{1/2}]$.

\section{Wrapping number}\label{secWrap}

In this section we will present two related methods that will help us determine the specific wrapping number of a given knot diagram in the solid torus ($ST$). These methods will be presented together with some basic results bounding the wrapping number of a knot. We will start by introducing the definition of wrapping number.

Let $K\subset ST$ be a knot in the solid torus, and let $\mathcal{D}(K)$ be the set of all regular diagrams of $K$ in the annulus. For a diagram $D\in \mathcal{D}(K)$, we define the \emph{wrapping number} of that specific diagram $D$ ---and write $w\subST (D)$--- as the minimal number of intersections a meridian of the annulus generates with $D$ when traversing from the center of the annulus to its exterior part. See \emph{Figure \ref{figMeri1}}.
\begin{figure}[ht]
	\centering
	\includegraphics[scale=0.5]{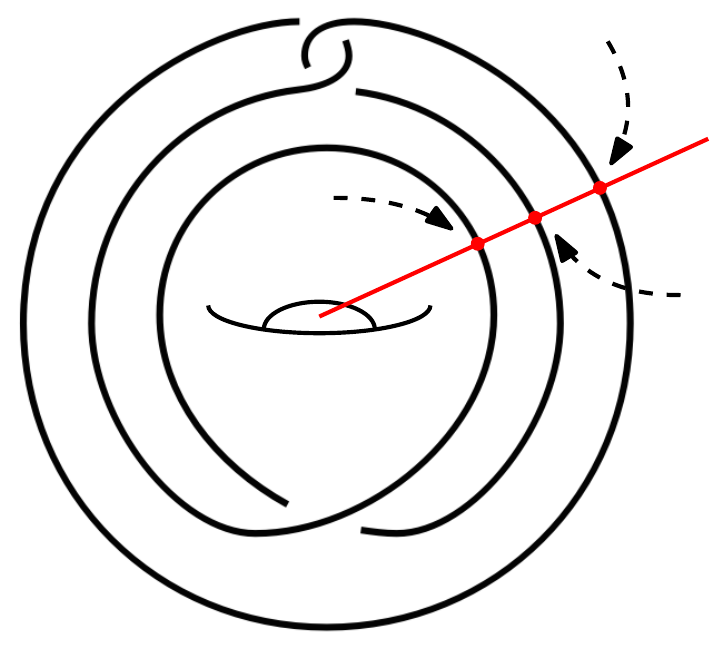}
	\caption{$w\subST(L(1,2))=3$.}
	\label{figMeri1}
\end{figure}

Using this definition, the wrapping number of a knot $K$ is then defined as
\[w\subST (K)=\min_{D\in \mathcal{D}(K)}w\subST (D).\]
For example, the unknot $U$ inside the solid torus has wrapping number $0$, since that is its minimal number of intersections it can have with a meridian. Its diagram $U_1$ in \emph{Figure \ref{figMeri2}} realizes this value. Nonetheless, it is worth seeing that the diagram of the unknot $U_2$ only has one meridian (up to isotopy) that crosses it, and it has wrapping number $2$ --- as diagram.
\begin{figure}[ht]
	\centering
	\includegraphics[scale=0.5]{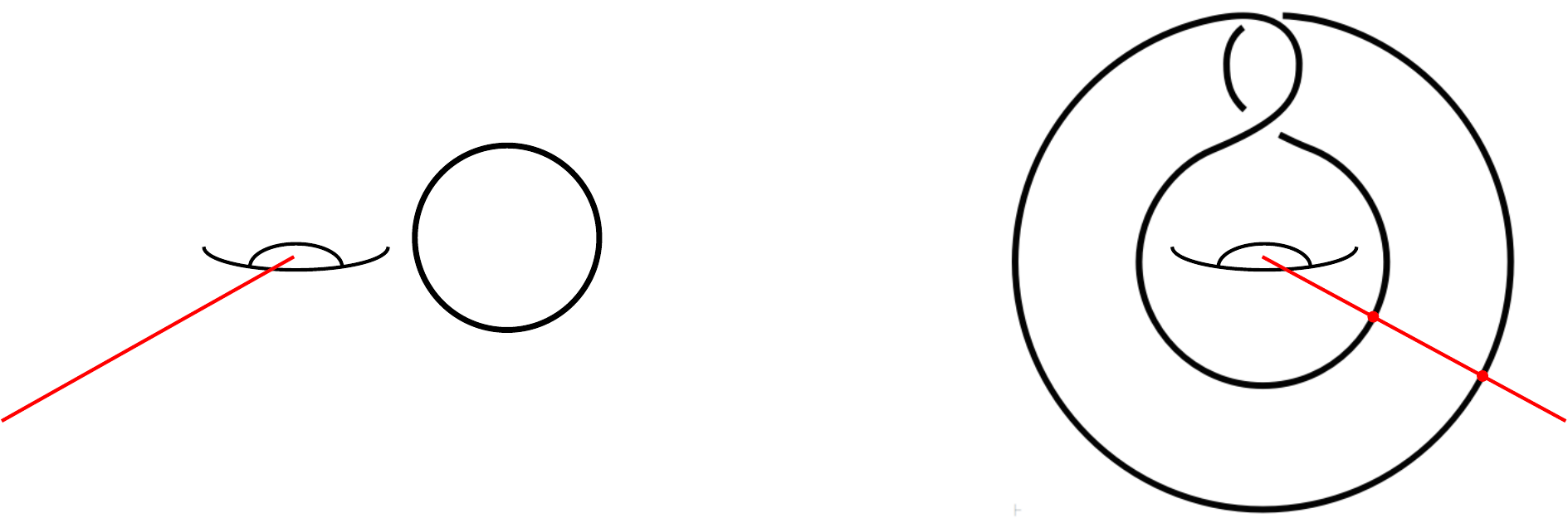}\\
	$U_1$\spa{20}\spa{20}\spa{14} $U_2$
	\caption{$w\subST(U)=w\subST(U_1)=0$, but $w\subST(U_2)=2$.}
	\label{figMeri2}
\end{figure}

In this work we do not require meridians to be a line segment, since when considering a knot with all its diagrams any diagram can be isotopically transformed so that the meridian crossing it would be straightened out, in case it is needed. Also, all isotopic meridians for a given diagram will be represented by only one representative at a time.

The \emph{wrapping number} (also known as \emph{geometric degree}) is not to be confused with the \emph{winding number}, which is defined in a similar fashion as the wrapping number, where each intersection of $D$ (oriented) with a meridian is counted with sign, and so $-w\subST (D)\leq wind (D)\leq w\subST (D)$. See \emph{Figure \ref{figWind}} for reference.
\begin{figure}[ht]
	\centering
	\includegraphics[scale=0.5]{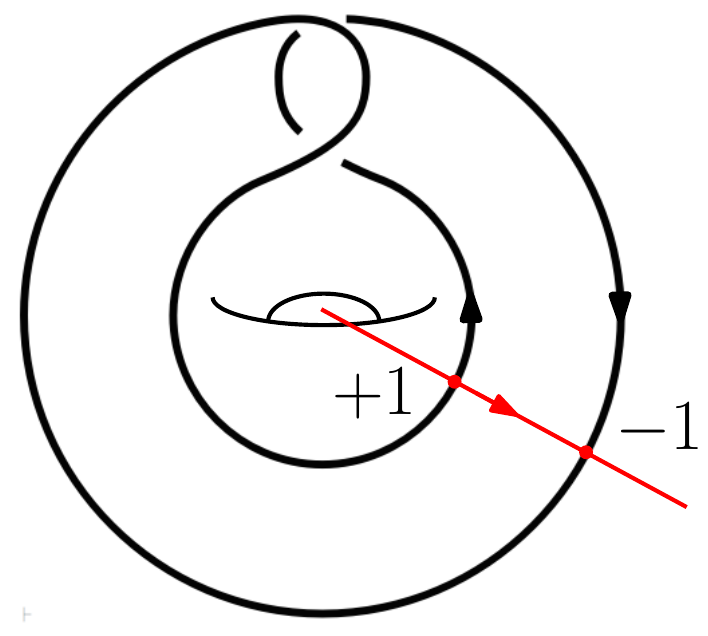}
	\caption{$w\subST(U_2)=2; wind(U_2)=0$.}
	\label{figWind}
\end{figure}

Let us now introduce the basic result that we will be using as a first bound setter for the wrapping number of a knot.

\begin{prop}\label{propWrap}
Let $K\subset ST$ be a knot in the solid torus with Jones polynomial $J\subST(K)=\sum_{i=0}^M\beta_iz\subST^i$, $\beta_M\neq 0$. Then
\[0\leq M\leq w\subST(K).\]
\end{prop}
\begin{proof}
Let $D$ be a diagram of $K$ such that it realizes the wrapping number of $K$, and suppose $M>w\subST(K)$. This $D$ would have a Kauffman bracket of the form $\kaufST{D}=\sum_{i=0}^M\alpha_iz\subST^i$, $\alpha_M\neq 0$. Since $z\subST^M\neq0$, it must arise from a state that encircles $M$ times the center of the annulus. But this is a contradiction, because if the center is surrounded $M$ times, that would be, at least, the amount of times a meridian would cross $D$ to reach its exterior part.
\end{proof}

\begin{mth}\label{method_1}
Let $K$ be a knot in $ST$. On the annular projection $D$ of $K$ we perform the following steps.
\begin{enumerate}
	\setcounter{enumi}{-1}
	\item Start at the center of the annulus.
	\item Color the region until we meet its bounds.
	\item Split the crossings in the boundary as shown:
	\[\includegraphics[scale=0.75]{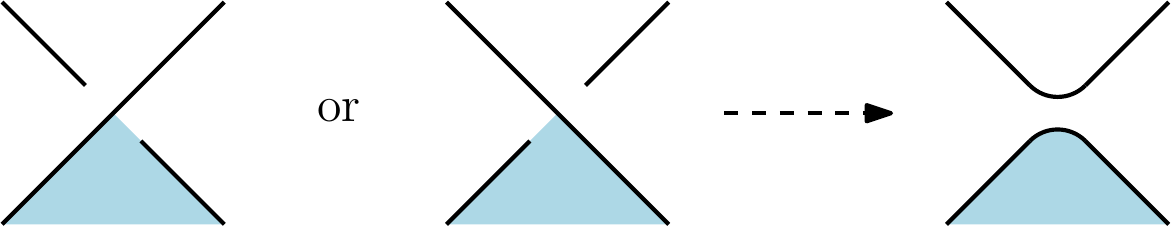}\]
	The colored region becomes bounded by a circuit that encircles the center of the annulus.
	\item Advance to the contiguous outer region. (We ignore any uncolored inner isolated regions that might have appeared.)
	\item
	\begin{itemize}
		\item If the region is bounded, repeat from \emph{Step 1}.
		\item If the region is unbounded, we represent the resulting diagram by $D^*$ and define
		\[ec(D^*):=\#\textnormal{\emph{colors}}\]
	\end{itemize}
\end{enumerate}
\end{mth}

Here $ec(D^*)$ stands for the ``encircling circuits (around the center of the annulus) of $D^*$'', and is defined as the number of colors used in the process of applying \emph{Method \ref{method_1}}.

\begin{exmp}\label{ex1}
We will use the following knot diagram that we will call $D_{20}$ throughout this section to illustrate how the methods here defined work.
\begin{figure}[H]
	\centering
	\includegraphics[scale=0.5]{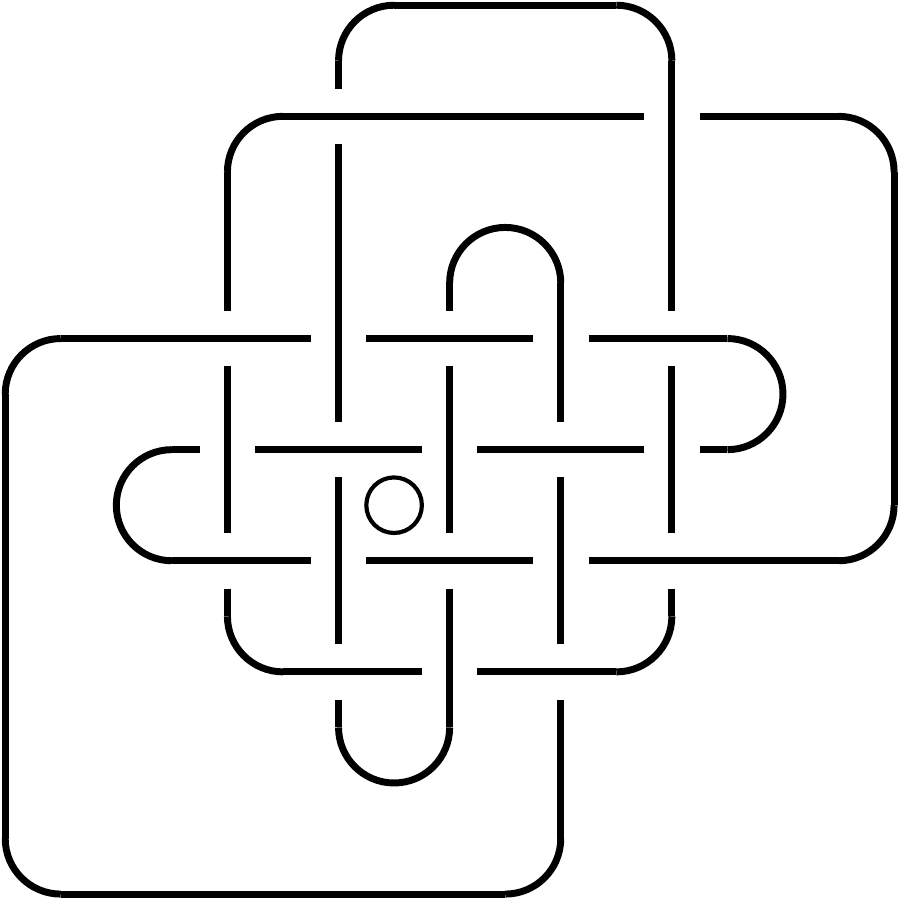}
	\caption{$D_{20}$.}
	\label{figD20}
\end{figure}
\noindent Let us apply \emph{Method \ref{method_1}} to $D_{20}$. As described, starting in the center of the annulus, we fill in with color the first region and split the crossings where the filling meets its bounds.
\begin{figure}[H]
	\centering
	\includegraphics[scale=0.575]{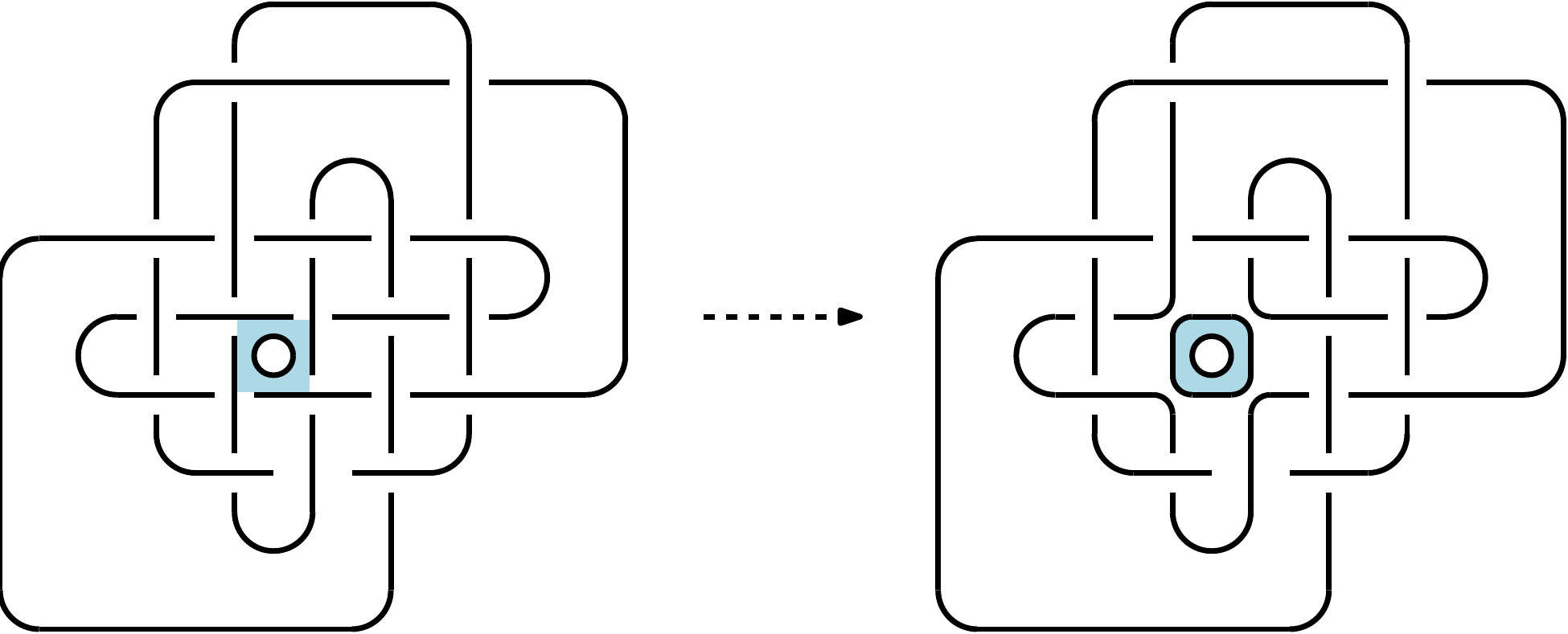}
	\caption{First application of \emph{Method \ref{method_1}}.}
	\label{figM11}
\end{figure}
\noindent Then, we proceed to the outer contiguous region and repeat the process.
\begin{figure}[H]
	\centering
	\includegraphics[scale=0.575]{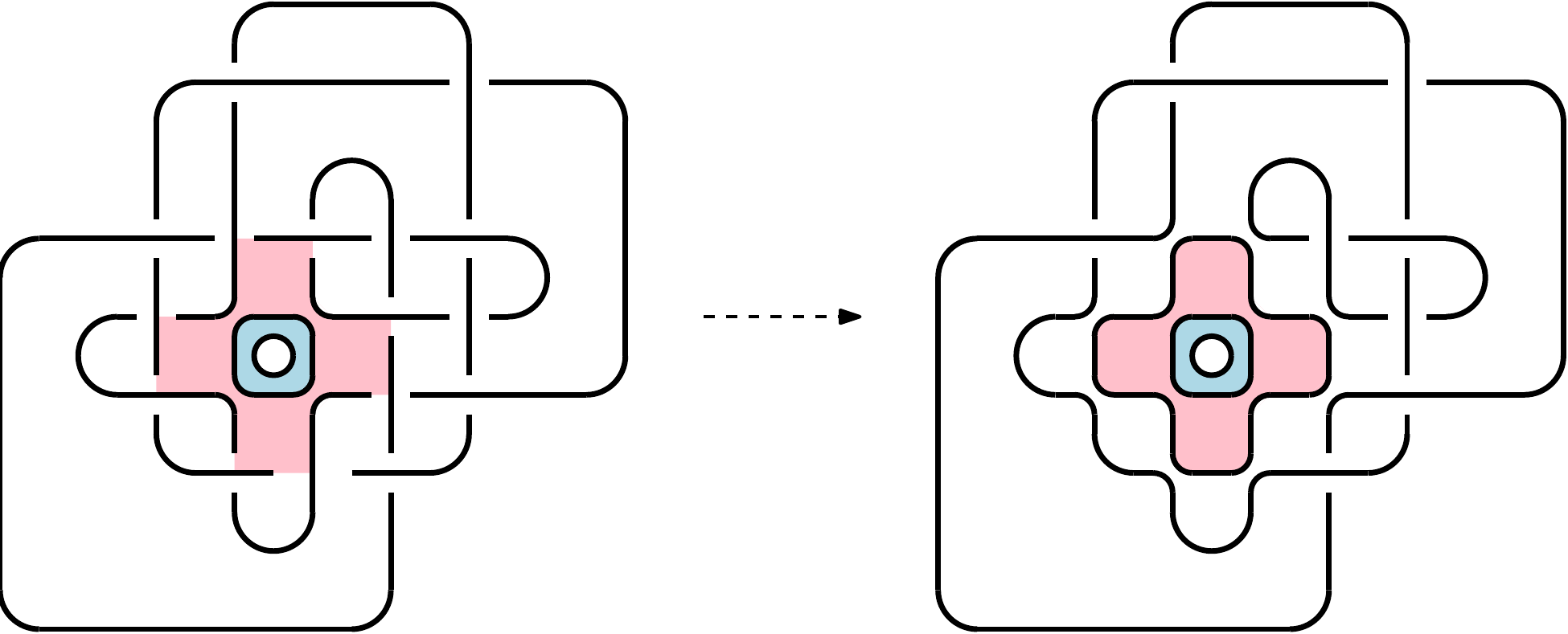}
	\caption{Second application of \emph{Method \ref{method_1}}.}
	\label{figM12}
\end{figure}
\noindent We do the same thing once again.
\begin{figure}[H]
	\centering
	\includegraphics[scale=0.575]{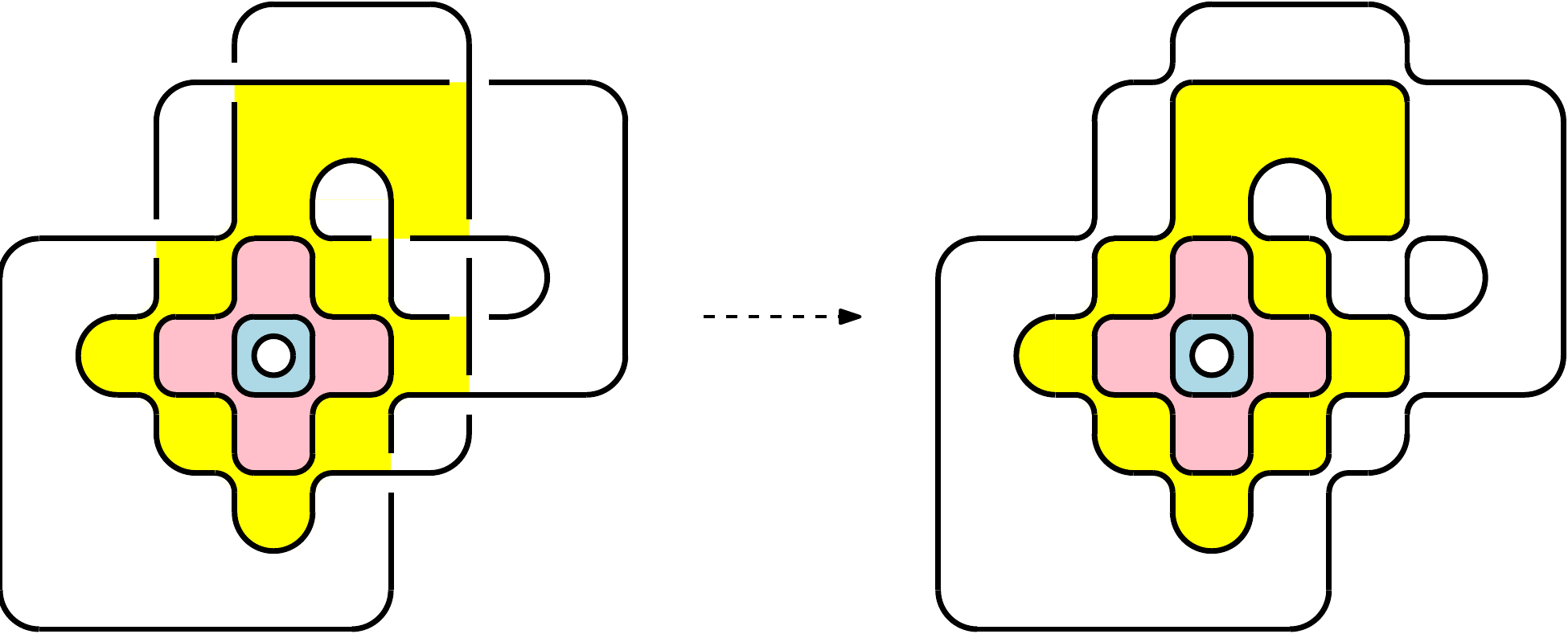}
	\caption{Third application of \emph{Method \ref{method_1}}.}
	\label{figM13}
\end{figure}
\vspace*{-0.5em}
\noindent And finally, in the last application there are no crossings to split, thus we just fill in the region with color. Please notice that as pointed out in \emph{Step 3}, there is an isolated region which needs not be taken into account (since it is irrelevant for the construction).
\begin{figure}[H]
	\centering
	\includegraphics[scale=0.48]{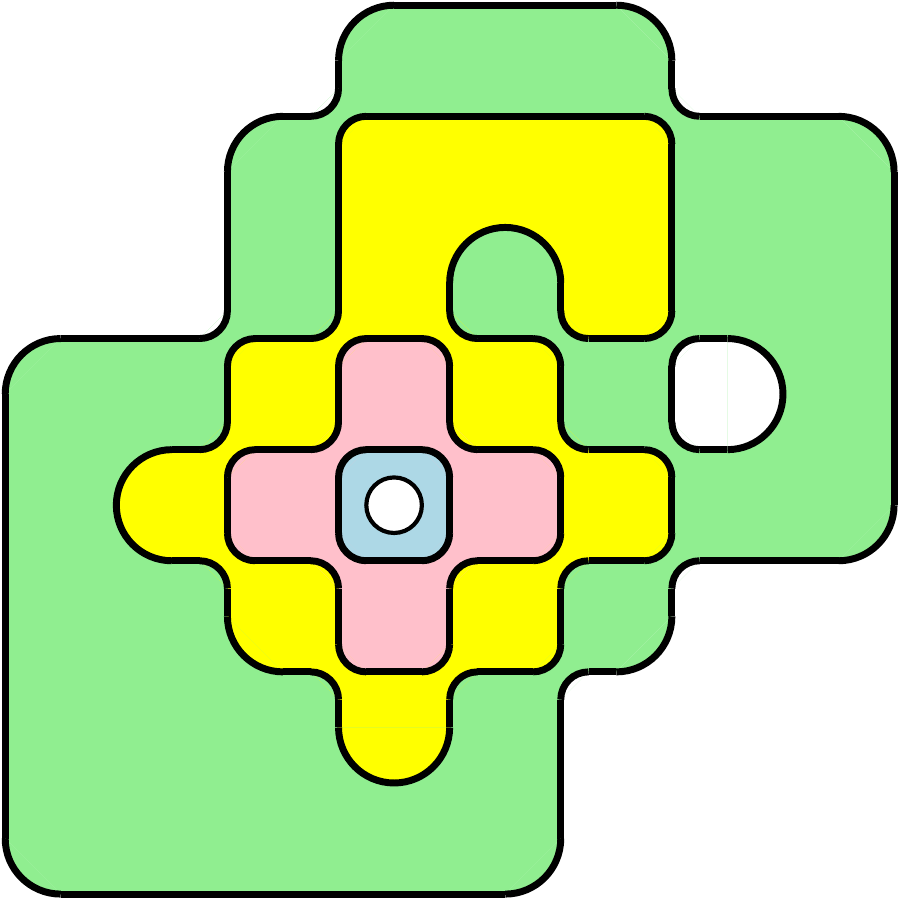}
	\caption{Last application of \emph{Method \ref{method_1}}.}
	\label{figM14}
\end{figure}
\vspace*{-0.5em}
\noindent If we expanded to yet the next outer region we would already be outside the knot, in an unbounded region, therefore the method has come to an end. Counting then the number of colors used in the whole process we can say that for this specific diagram
\[ec(D_{20}^*)=4.\]
\end{exmp}
As seen in the example, this $ec(D^*)$ value basically measures the number of times the knot projection is revolving around the center of the annulus. We will prove this value is actually the minimal wrapping number of $D$. This is better converted into words in the following result.

\begin{mth}\label{method_2}
Let $K$ be a knot in $ST$ and $D$ an annular diagram of $K$. On $D^*$, obtained through \emph{Method \ref{method_1}}, we perform the following steps.
\begin{enumerate}
	\setcounter{enumi}{-1}
	\item Start at the center of the annulus.
	\item Connect the colored region with its contiguous colored (outer) next through all connecting walls using directed edges.
	\item Advance to the contiguous outer region.	\item
	\begin{itemize}
		\item If the region is bounded, repeat from \emph{Step 1}.
		\item If the region is unbounded, any reversed path from the exterior of the projection to the center of the annulus is a valid meridian.
	\end{itemize}
\end{enumerate}
\end{mth}

Please notice that all reverse paths starting at the outside of $D$ lead to the center of the annulus, since there are no more pure source regions (only outbound edges) than the center of the annulus---any other region either has in and out edges or is a sink.

Let us illustrate how this method works in a specific example, so that the general idea of what is going on is better grasped.

\begin{exmp}
We will use the same diagram as in \emph{Figure \ref{figD20}}. Since this method starts with the result of applying \emph{Method \ref{method_1}}, we start \emph{Step 0} with $D_{20}^*$.
\begin{figure}[H]
	\centering
	\includegraphics[scale=0.575]{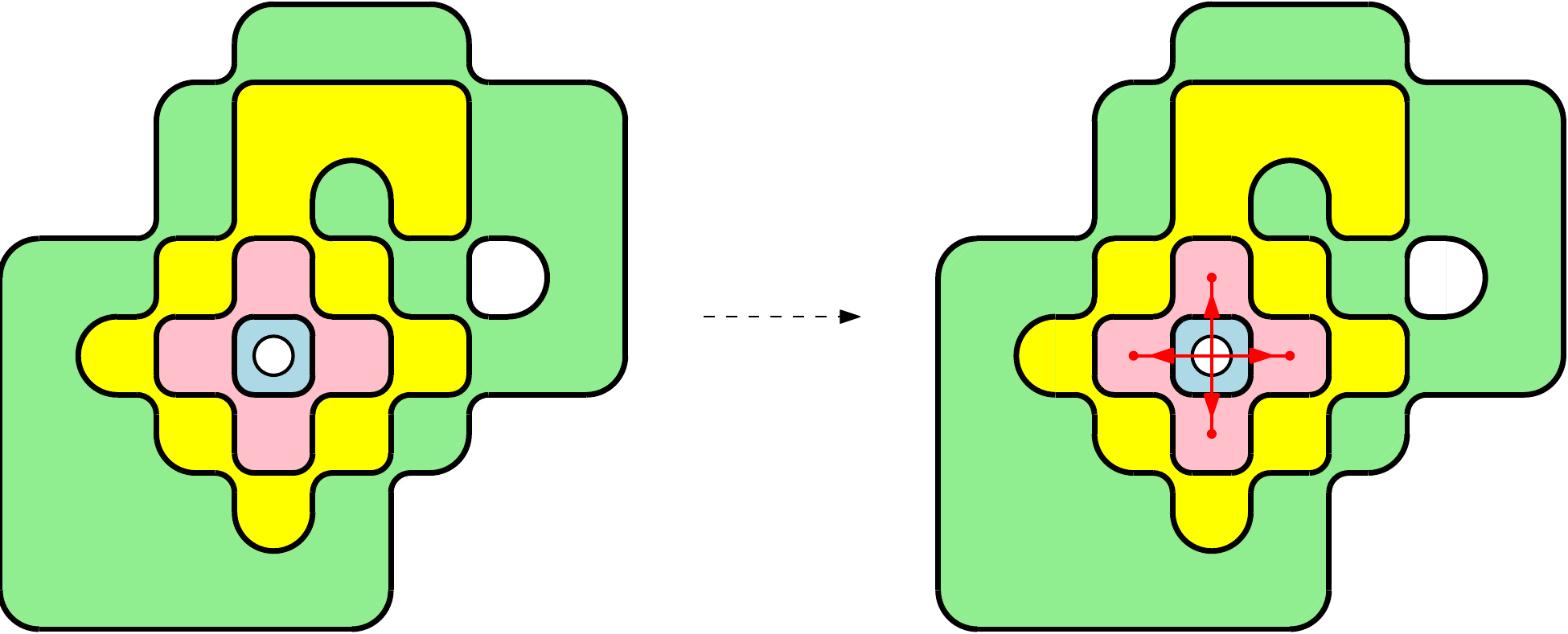}
	\caption{First application of \emph{Method \ref{method_2}}.}
	\label{figM21}
\end{figure}
And then keep repeating the process until we reach the outmost region.
\begin{figure}[H]
	\centering
	\includegraphics[scale=0.575]{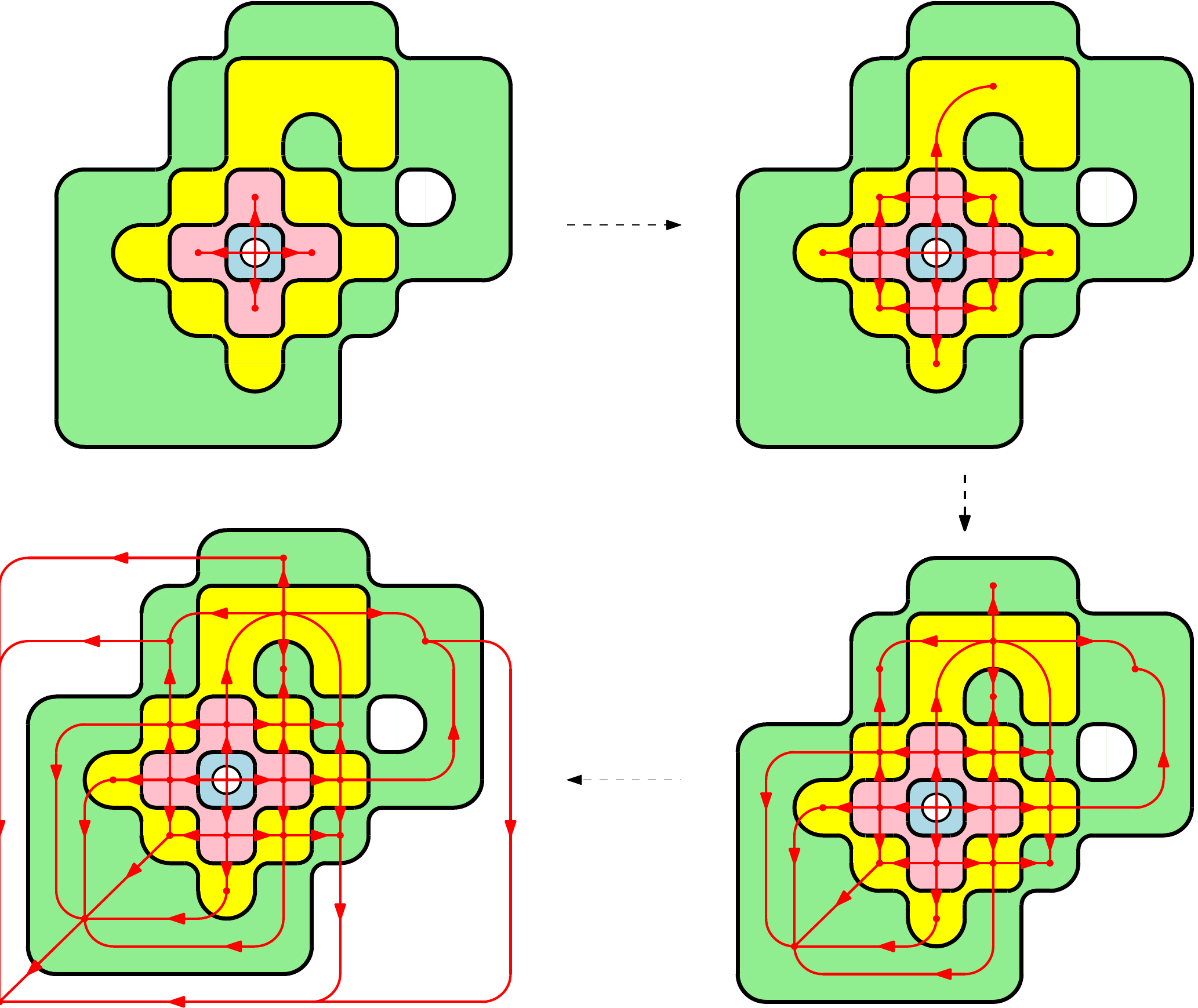}
	\caption{\emph{Method \ref{method_2}} overview.}
	\label{figM22}
\end{figure}
\noindent At this moment we have come to \emph{Step 3} with an unbounded region, thus any path going backwards from the outmost region point to the center of the annulus will give us a meridian that crosses $D_{20}$ as many times as $w\subST(D)$. \emph{Figure \ref{figM23}} shows multiple valid choices of this meridian.
\begin{figure}[H]
	\centering
	\includegraphics[scale=1]{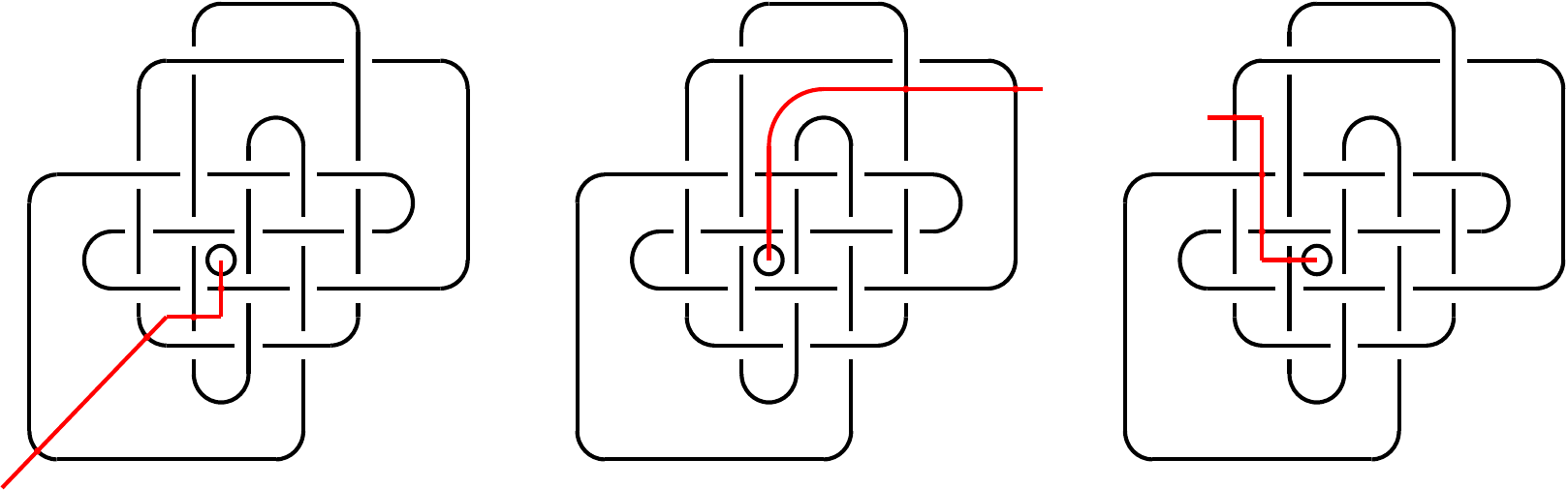}
	\caption{Possible meridians that give rise to $w\subST(D)$.}
	\label{figM23}
\end{figure}
\end{exmp}

\begin{thm}
Let $D$ be a knot diagram in the annulus, and apply \emph{Methods \ref{method_1}} and \emph{\ref{method_2}} to it. Then, any reversed path from the exterior of the projection to the center of the annulus is a minimal meridian cutting $D$ at $w\subST(D)$ points.
\end{thm}
\begin{proof}
The proof is constructive and inductive. Let us consider a meridian $m$ that will start at the center of the annulus. This meridian will have a ``basis'' ---center of the annulus--- and a ``head'' ---tip of the meridian looking forward to get out of $D$. For $m$ to reach the exterior of $D$, it will first need to ``hit'' the knot's innermost strands. This is, there is a region delimited by the innermost strands and crossings of the diagram that delimits the area of free movement of the ``head'' of $m$ before it hits its first wall---blue region in \emph{Figure \ref{figM11}}. Once this happens, this will count as $+1$ crossings between the meridian and $D$, after which it would enter a new region---see the meridians of \emph{Figure \ref{figM21}}. We set the ``head'' of $m$ in all possible regions where it could have progressed when advancing from the previous region, and label the region that it just left with a color. From here on the proof is inductive: we will repeat the process of searching for the new strands surrounding every possible ``head'' of the meridian as it advances towards the exterior part. If during the process we encounter a bounding region which has already been colored, it would mean that the fastest path to it is not the one we are heading, but the one which generated its coloring. Therefore we do not create a directed edge (meridian) between already colored regions. The process finishes when a meridian reaches the exterior of $D$. The result would look similar to the last step of \emph{Figure \ref{figM22}}.\\
\noindent Any path that reaches the exterior in these steps will be minimal in crossings with $D$---which will be exactly $ec(D^*)$ times---, and since there is no other pure source (only edges pointing outwards) than the center of the annulus any path reversed is a valid meridian.
\end{proof}

\begin{corol}\label{corol_ec}
$w\subST(D)=ec(D^*).$
\end{corol}
This result is extracted from the previous proof, and it helps us determine the wrapping number of a diagram without the need to explicitly find a meridian that realizes it---it suffices to apply \emph{Method 1}.

\begin{exmp}
Applying \emph{Method \ref{method_1}} to the \emph{lasso} family $L(r_1,...,r_m)$ with its usual projection $D_{L(r_1,...,r_m)}$ we get $ec(D_{L(r_1,...,r_m)})=m+1$.
\begin{figure}[H]
	\centering
	\includegraphics[scale=0.5]{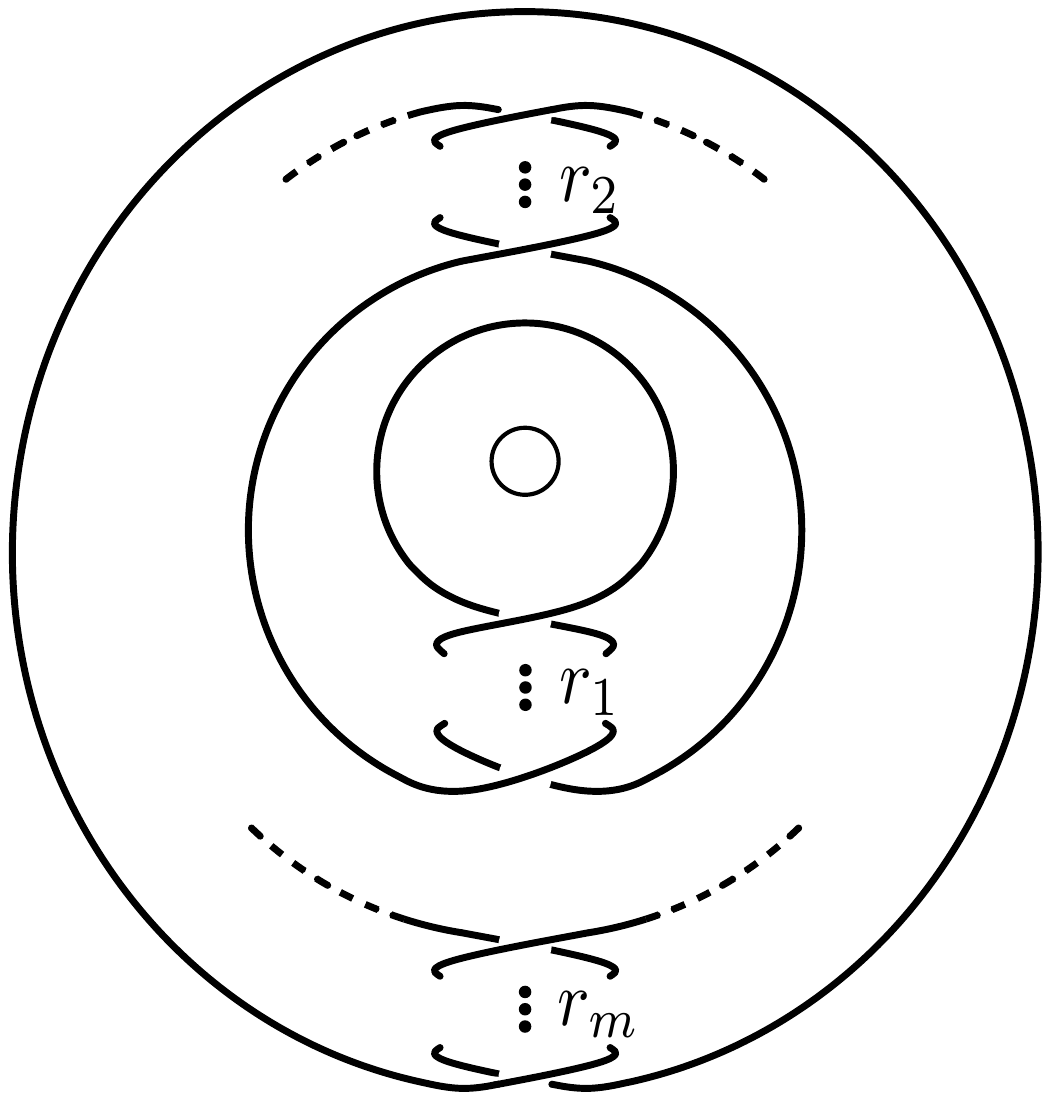}
	\caption{Standard diagram of the lasso $L(r_1,...,r_m)$ in the annulus.}
	\label{figLassoWrap}
\end{figure}
\noindent We now want to calculate the value of $M$ in $J\subST(L(r_1,...,r_m))=\sum_{i=0}^M\beta_iz\subST^i$, or, equivalently, the value of $M$ in $\kaufST{D_{L(r_1,...,r_m)}}=\sum_{i=0}^M\alpha_iz\subST^i$. The general calculation is not easily manageable, but in this occasion we will restrict the lasso to the case $r_1,...,r_m>0$ (the same can be done if all values are negative). In this case, $D_{L(r_1,...,r_m)}$ is minus-adequate (see \emph{Section \ref{secAdequacy}} for its definition), and the polynomial accompanying the state is
\[\kaufST{D_{L(r_1,...,r_m)}|s_-}=A^{\sum_{i=1}^mr_i}(-A^{-2}-A^2)^{\sum_{i=1}^m(r_i-1)}z\subST^{m+1},\]
term which is not cancelled out by any other in $\kaufST{D_{L(r_1,...,r_m)}}$ since $D_{L(r_1,...,r_m)}$ is minus-adequate. Therefore
\[M=m+1.\]
\noindent Now, using \emph{Corollary \ref{corol_ec}} and applying \emph{Proposition \ref{propWrap}} we get
\[m+1=M\leq w\subST(L(r_1,...,r_m))\leq w\subST(D_{L(r_1,...,r_m)})=ec(D_{L(r_1,...,r_m)})=m+1,\]
thus $w\subST(L(r_1,...,r_m))=m+1$ is the wrapping number of the lasso (not just of a diagram of it).
\end{exmp}

Furthermore, \emph{Corollary \ref{corol_ec}} can be extendedly used to calculate the geometrical intersection degree of two-component links where one of the components is the unknot, as we can see in the following example.

\begin{exmp}
Let us consider the link in \emph{Figure \ref{figMotegi}}, which can be found in \cite{Mo}.
\begin{figure}[H]
	\centering
	\includegraphics[scale=0.25]{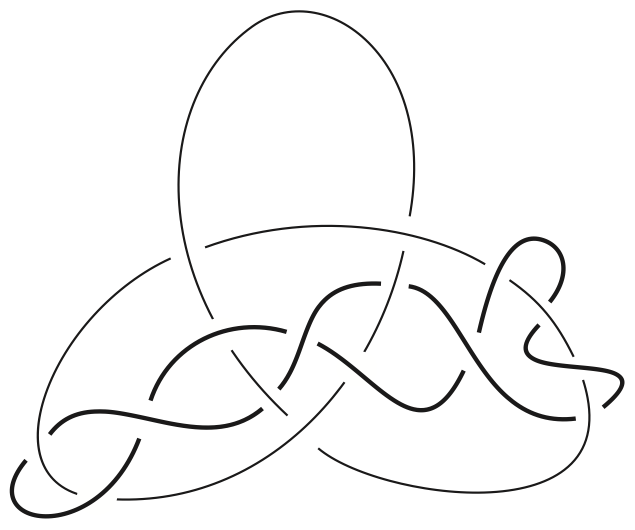}
	\caption{Link diagram of the trefoil $3_1$ with an unknot projection.}
	\label{figMotegi}
\end{figure}
\noindent The diagram can be deformed so that the unknot component appears without any crossing, and then we can push the knotted component further so that the unknot component shows itself as the boundary of a meridian of the solid torus.
\begin{figure}[H]
	\centering
	\includegraphics[scale=0.25]{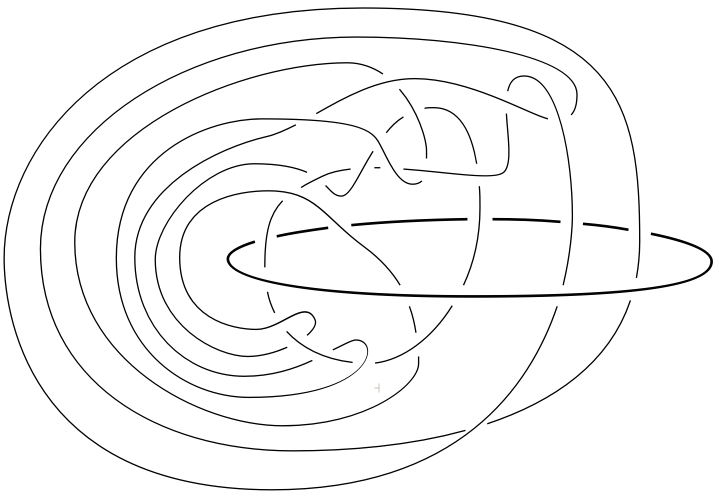}
	\caption{Link diagram with the unknot component shaped as a meridian of $ST$.}
	\label{figMot1}
\end{figure}
\noindent Let $D$ be the knotted component of the link diagram shown in \emph{Figure \ref{figMot1}}. In these conditions we can apply \emph{Method \ref{method_1}}---considering the center of the annulus on the left-hand side of the meridian---and \emph{Corollary \ref{corol_ec}} to extract the wrapping number of $D$, and using \emph{Proposition \ref{propWrap}} we get
\[w\subST(D)=5,\]
which, going back to the original framework, tells us that the geometrical intersection degree of the trefoil and the unknot projection is $5$.
\end{exmp}


\section{Adequacy}\label{secAdequacy}

In this section we will recall from other authors' works or prove if necessary all the results regarding the adequacy of links that we will use in the subsequent sections of this work. For starters, let us first recall what adequacy itself is.

Let $L$ be a link, and $D$ be an $n$-crossing diagram ($n\geq1$) of $L$ with crossings labelled: $c_1,c_2,...,c_n$. A \emph{state} of $D$ is a function $s:\{c_1,c_2,...,c_n\}\to\{-1,1\}$. The effect of applying this function on the $i^{th}$ crossing of $D$ is illustrated in \emph{Figure \ref{figState}}.
\begin{figure}[ht]
	\centering
	\includegraphics[scale=0.75]{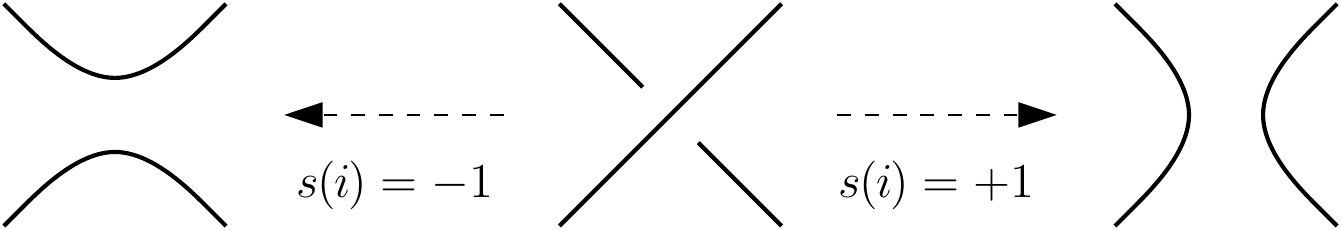}
	\caption{$i^{th}$ crossing of $D$.}
	\label{figState}
\end{figure}

When this function is applied to all crossings, the resulting $sD$ is a set of disjoint simple closed curves (\emph{circuits}) with no crossings. We denote the number of these circuits by $|sD|$.
The Kauffman bracket of a knot diagram $D$ can be expressed in these terms as
\[\kauf{D}=\sum_s\Big(A^{\sum_{i=1}^ns(i)}(-A^{-2}-A^2)^{|sD|-1}\Big),\]
where $s$ covers all possible states of $D$.

Let now $s_+$ and $s_-$ be the two extremal states that satisfy $\sum_{i=1}^ns_+(i)=n$ and $\sum_{i=1}^ns_-(i)=-n$. We call a diagram $D$ \emph{plus-adequate} if $|s_+D|>|sD|$ for all $s$ with $\sum_{i=1}^ns(i)=n-2$, and \emph{minus-adequate} if $|s_-D|>|sD|$ for all $s$ with $\sum_{i=1}^ns(i)=-n+2$. This means that, for plus-adequate, for every state taken which has all-but-one crossings positively split (and one negative), the number of circuits generated is strictly smaller than the number of circuits generated by cutting all crossings positively. Visually, this is to say that every crossing in $D$ was split in such a way that each side of the splitting belongs to a different circuit. Conversely, if a circuit abuts itself, it is not plus-adequate. The same holds for minus-adequate. 
If both conditions are satisfied at once $D$ is called \emph{adequate}, and any link having an adequate diagram will also be called adequate. Adequate links are remarkable for being, in particular, an extension of the alternating links family \cite{Li1,Li2} (alternating and reduced implies adequate). A link will be called adequate if there exists an adequate diagram for it. The trivial knot will not be considered adequate. \emph{Figures \ref{figAdeqs}}, \emph{\ref{figAdeqs2}} and \emph{\ref{figAdeqs3}} respectively show several alternating adequate prime knots, the first prime knots not to be adequate ($8_{19}$, $8_{20}$ and $8_{21}$) and the first prime knots to be adequate but not alternating ($10_{152}$, $10_{153}$ and $10_{154}$) \cite{Kh}.
\begin{figure}[ht]
	\centering
	\includegraphics[scale=0.15]{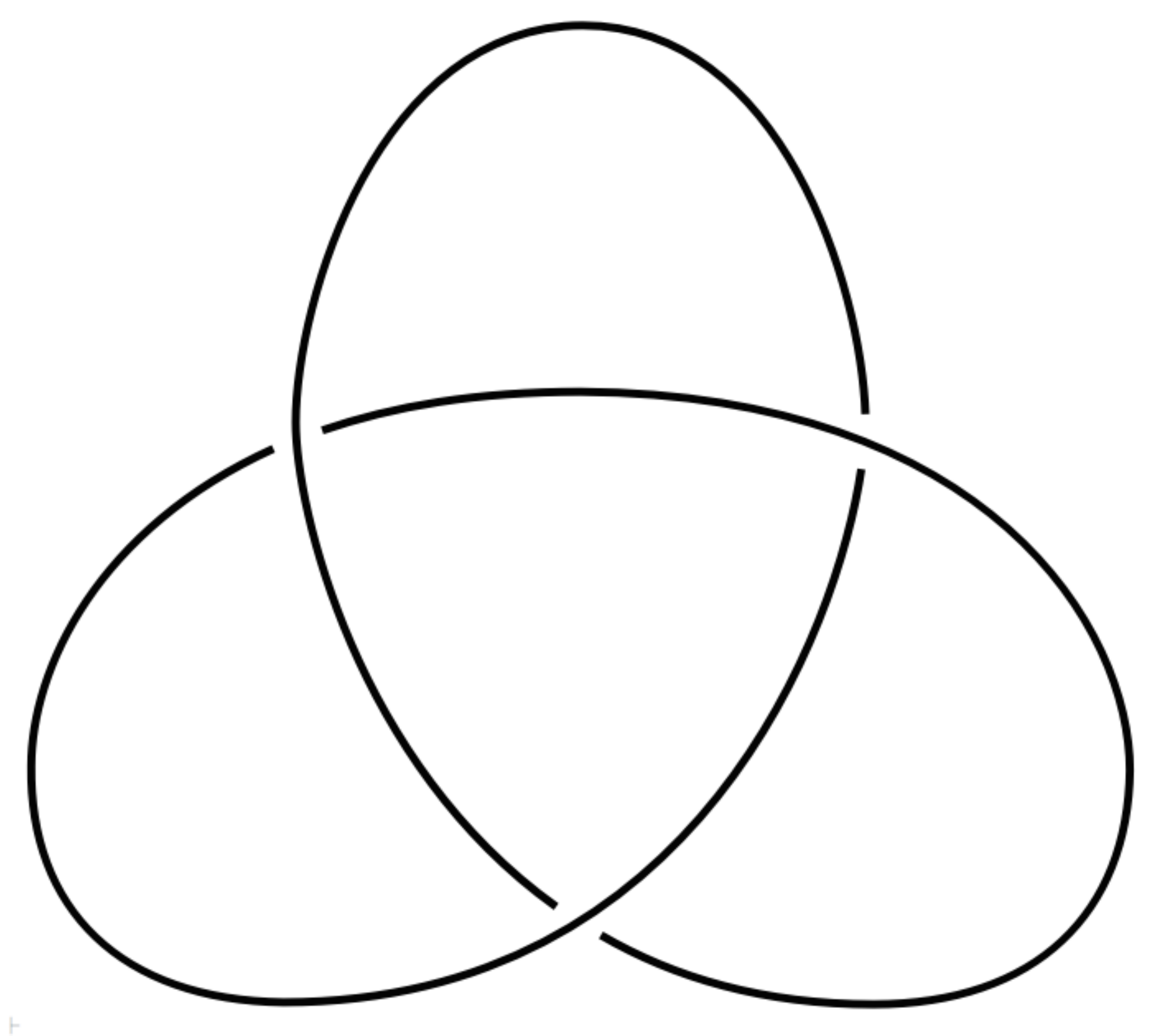}
	\spa{20}
	\includegraphics[scale=0.175]{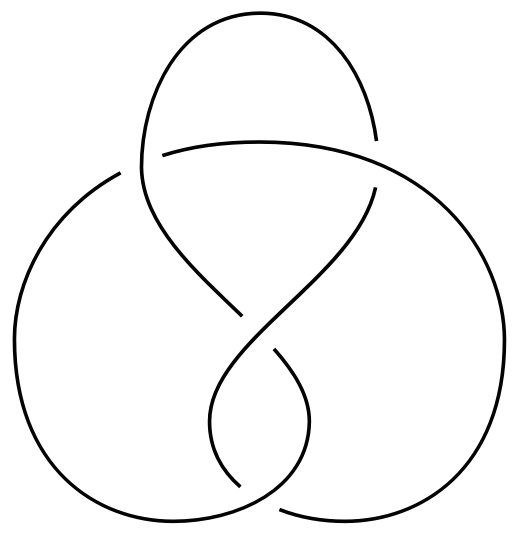}
	\spa{20}
	\includegraphics[scale=0.15]{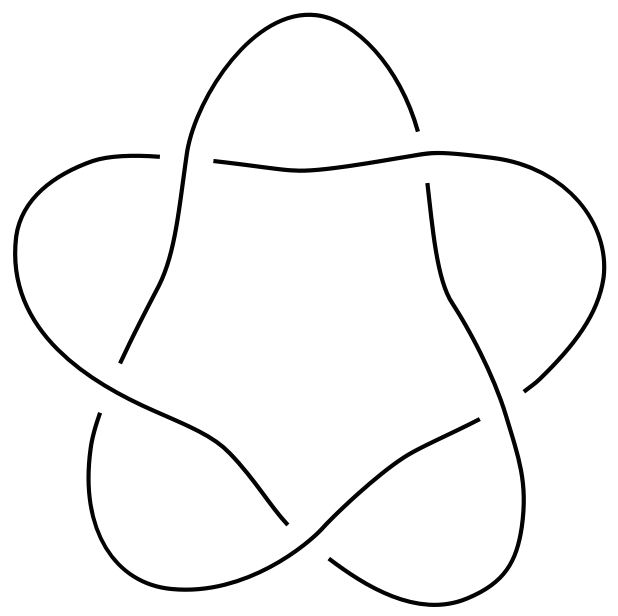}\\
	$\spa{4}3_1$\spa{20}\spa{20}\spa{14}$4_1$\spa{20}\spa{20}\spa{11}$5_1$
	\caption{First alternating adequate knots.}
	\label{figAdeqs}
\end{figure}
\begin{figure}[ht]
	\centering
	\includegraphics[scale=0.18]{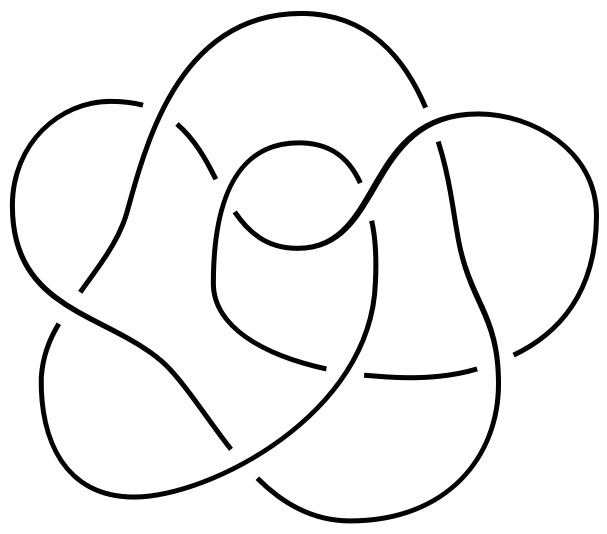}
	\spa{15}
	\includegraphics[scale=0.18]{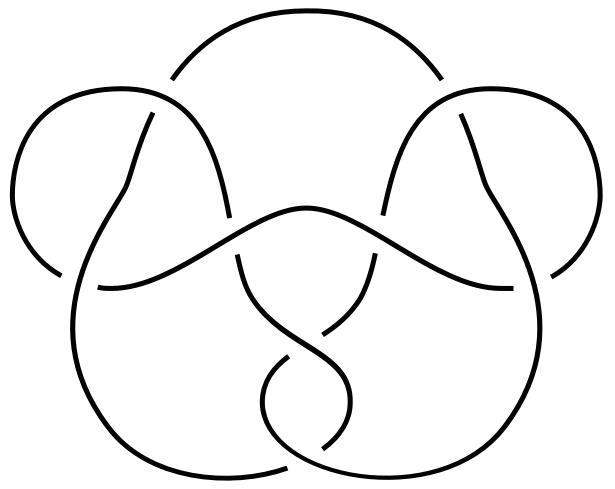}
	\spa{15}
	\includegraphics[scale=0.18]{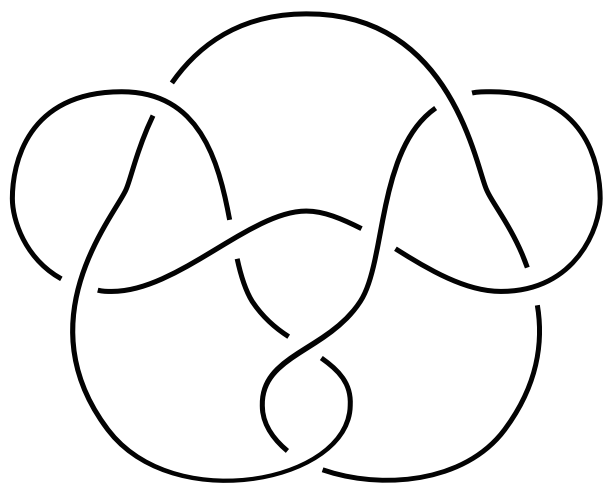}\\
	$8_{19}$\spa{20}\spa{20}\spa{12}$8_{20}$\spa{20}\spa{20}\spa{13}$8_{21}$
	\caption{First non-adequate knots.}
	\label{figAdeqs2}
\end{figure}
\begin{figure}[ht]
	\centering
	\includegraphics[scale=0.18]{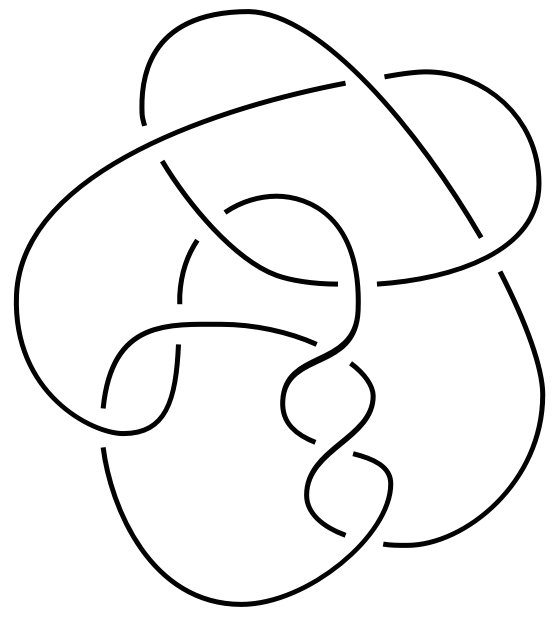}
	\spa{20}
	\includegraphics[scale=0.18]{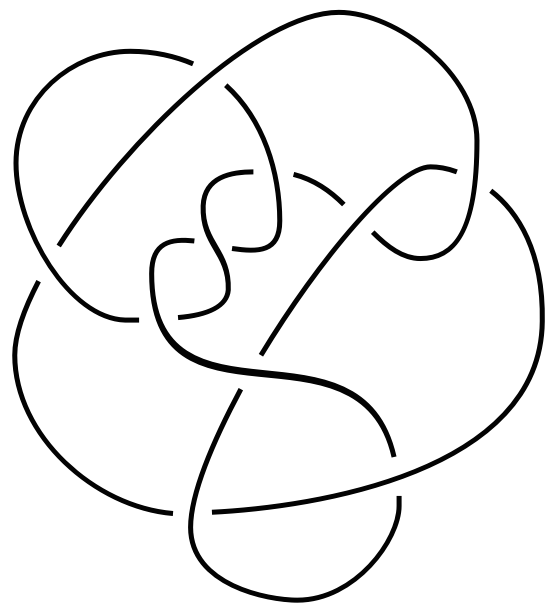}
	\spa{20}
	\includegraphics[scale=0.18]{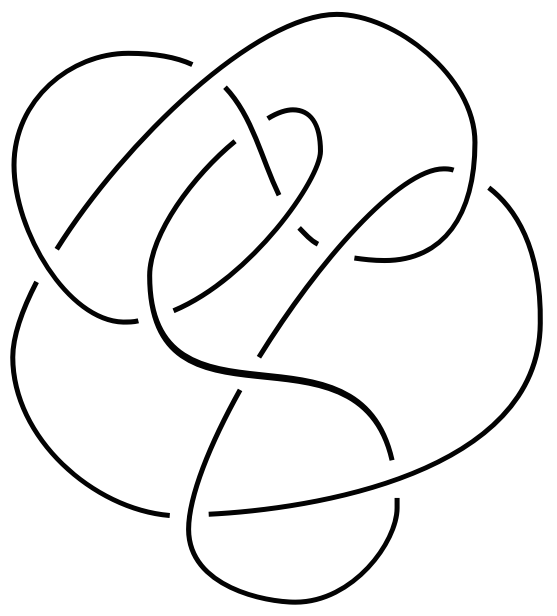}\\
	$\spa{4}10_{152}$\spa{20}\spa{20}\spa{7}$10_{153}$\spa{20}\spa{20}\spa{10}$10_{154}$
	\caption{First non-alternating adequate knots.}
	\label{figAdeqs3}
\end{figure}

We will now present several well known results. Some of them can be explicitly found in the literature \cite{Li1}, while others are usually assumed but nonetheless we will name and prove them here.

\begin{lmm}\label{lemmaLick4}
Let $D$ be a link diagram with $n$ crossings. Then
\begin{enumerate}[label=(\roman*)]
	\item $M_{\kauf{D}}\leq n+2|s_+D|-2$, with equality if $D$ is plus-adequate, and
	\item $m_{\kauf{D}}\geq -n-2|s_-D|+2$, with equality if $D$ is minus-adequate.
\end{enumerate}
\end{lmm}

Here and throughout this work we will use $\kauf{D}$ as usually to express the Kauffman bracket of $D$, and $M_p$ and $m_p$ will represent respectively the maximum and minimum exponents of the indeterminate of a Laurent polynomial $p$. In this work, this Laurent polynomial will always be either a Kauffman bracket or a Jones polynomial, therefore $p\in\mathbb{Z}[A^{-1},A]$ with indeterminate $A$ or $p\in\mathbb{Z}[t^{-1/2},t^{1/2}]$ with indeterminate $t$.

\begin{proof}
As previously pointed out, the Kauffman bracket of $D$ can be expressed as in the following formula:
\[\kauf{D}=\sum_s\Big(A^{\sum_{i=1}^ns(i)}(-A^{-2}-A^2)^{|sD|-1}\Big).\]
In particular, the part of that polynomial generated by the state $s_+$ is
\[\kauf{D|s_+}=A^{\sum_{i=1}^ns_+(i)}(-A^{-2}-A^2)^{|s_+D|-1}=A^{n}(-A^{-2}-A^2)^{|s_+D|-1},\]
and so $M_{\kauf{D|s_+}}=n+2|s_+D|-2$. Now let us consider $s_1$ a state that differs from $s_+$ in one crossing (cut negatively). For this state, the polynomial generated looks like
\[\kauf{D|s_1}=A^{\sum_{i=1}^ns_1(i)}(-A^{-2}-A^2)^{|s_1|-1}=A^{n-2}(-A^{-2}-A^2)^{|s_1|-1}.\]
Since only one crossing was cut differently from $s_+$, $\sum_{i=1}^ns_1(i)=n-2$. In general, we observe that for every crossing cut that differs from $s_+$, the sum $\sum_{i=1}^ns(i)$ decreases by $2$. This is, if $s_r$ is a state that differs in $r$ crossing cuts from $s_+$ the count for its state value will be $\sum_{i=1}^ns_r(i)=n-2r$, being the extreme case $s_-=s_n$ (all crossings cut differently) where $\sum_{i=1}^ns_-(i)=-n$. Now, watching carefully the amount of circuits generated in any $s_r$ and $s_{r+1}$, we appreciate that by just changing one crossing cut all circuits remain untouched except for the vicinity of a crossing which would either merge two circuits into one or divide one circuit into two. In other words, $|s_r|=|s_{r+1}|\pm1$.
\begin{figure}[ht]
	\centering
	\includegraphics[scale=0.5]{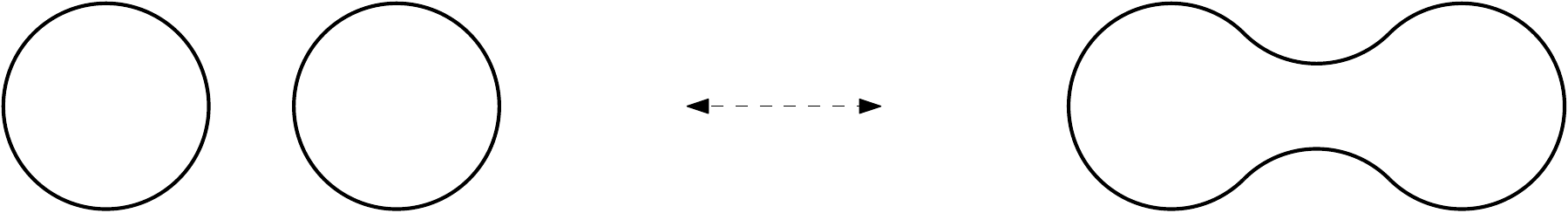}
\end{figure}

\noindent Hence, when considering the maximum exponent of their Kauffman brackets we get
\begin{align*}
M_{\kauf{D|s_r}}-M_{\kauf{D|s_{r+1}}}&=n-2r+2|s_rD|-2-(n-2(r+1)+2|s_{r+1}D|-2)\\
&=2+2(|s_rD|-|s_{r+1}D|)=2\pm2\geq0.
\end{align*}
In other words, the maximum exponent of the Kauffman bracket of $D$ never increases as we cut more crossings negatively. Lastly, if $D$ is adequate by definition we know that $M_{\kauf{D|s_+}}>M_{\kauf{D|s_1}}$, and so the contributing part of $\kauf{D|s_+}$ in the global Kauffman bracket will not be cancelled out and thus $M_{\kauf{D}}\leq n+2|s_+D|-2$, with equality if $D$ is plus-adequate as shown above. An analogous argument can be used for $m_{\kauf{D}}$.
\end{proof}

\begin{corol}\label{corolLick5}
If $D$ is an adequate diagram, then
\[M_{\kauf{D}}-m_{\kauf{D}}=2n+2|s_+D|+2|s_-D|-4.\]
\end{corol}

For the following lemma, let us recall that a link diagram is called \emph{connected} if it forms a connected subset of the plane when drawn as a flat projection of the link --- without ``over'' and ``under'' crossings.

\begin{lmm}\label{lemmaLick6}
Let $D$ be a connected link diagram with $n$ crossings. Then
\[|s_+D|+|s_-D|\leq n+2.\]
If $D$ is alternating, the equality is attained.
\end{lmm}

As in other works, we will use $B(p)$ to express the breadth of a Laurent polynomial $p$. More specifically, $B(p)=M_p-m_p$. As for the Jones polynomial of a knot $K\subset\mathbb{S}^3$ we will use the notation $J(K)$.

\begin{prop}\label{propBreadth}
Let $D$ be a link diagram of an oriented link $L$. Then
\[B(J(L))=\frac{B(\kauf{D})}{4}.\]
\end{prop}
\begin{proof}
The Jones polynomial of an oriented link $L$ can be calculated from its diagram $D$ in terms of the Kauffman bracket of $D$ as
\[J(L)=(-A^{-3})^{w(D)}\kauf{D}\Big|_{t^{-1/2}=A^2},\]
where $w(D)$ is the \emph{writhe} of $D$.
Then, its breadth can also be easily calculated in terms of the breadth of the Kauffman bracket as shown beneath, taking good care of the breadth calculation change when the respective polynomials are $\kauf{D}\in\mathbb{Z}[A^{-1},A]$ and $J(L)\in\mathbb{Z}[t^{-1/2},t^{1/2}]$. For this purpose, we need to notice that when the substitution $t^{-1/2}=A^2$ is performed, the positive exponents of $\kauf{D}$ become $4$ times smaller, and negative; as well as the negative exponents become $4$ times smaller, and positive. Therefore, due to this change in sign, the maximum exponent of $B(J(L))$ will arise from $m_{\kauf{D}}$ and the minimum from $M_{\kauf{D}}$.
\begin{align*}
B(J(L))&=M_{J(L)}-m_{J(L)}=\Big[-3w(D)+m_{\kauf{D}}-(-3w(D)+M_{\kauf{D}})\Big]_{t^{-1/2}=A^2}\\
&=\Big[m_{\kauf{D}}-M_{\kauf{D}}\Big]_{t^{-1/2}=A^2}=\Big[m_{\kauf{D}}\Big]_{t^{-1/2}=A^2}-\Big[M_{\kauf{D}}\Big]_{t^{-1/2}=A^2}\\
&=\frac{-m_{\kauf{D}}}{4}-\frac{-M_{\kauf{D}}}{4}=\frac{M_{\kauf{D}}-m_{\kauf{D}}}{4}=\frac{B(\kauf{D})}{4}.
\end{align*}
\end{proof}

Lastly for this ``previous results reminder'', we present the result which will be the cornerstone for the forthcoming part of the work. This is a very well-known result.
\begin{thm}[\cite{Li1}, Theorem 5.9]\label{keythm}
Let $D$ be a connected, $n$-crossing diagram of an oriented link $L$. Then
\begin{enumerate}[label=(\roman*)]
	\item $B(J(L))\leq n$;
	\item if $D$ is alternating and reduced, then $B(J(L))=n$.
\end{enumerate}
\end{thm}
The first result is a direct consequence of \emph{Lemma \ref{lemmaLick4}}, \emph{Lemma \ref{lemmaLick6}} and \emph{Proposition \ref{propBreadth}}. In particular, we can extend the inequality to $B(J(L))\leq c(L)$ --- where $c(L)$ is the \emph{minimal crossing number} of $L$ --- since the result applies to all diagrams of $L$. The equality in \emph{(ii)} makes use of the fact that alternating and reduced is adequate.

With these results as background, we continue to introduce a couple of basic results which are closely related to these.

\begin{lmm}\label{lemmaGeq}
Let $D$ be an adequate diagram. Then
\begin{enumerate}[label=(\roman*)]
	\item $|s_+D|\geq 2$;
	\item $|s_-D|\geq 2$.
\end{enumerate}
\end{lmm}
\begin{proof}
Since $D$ is adequate, by definition $|s_+D|>|sD|$ for all $s$ with $\sum_{i=1}^ns(i)=n-2$, assuming $D$ has $n$ crossings. Since $|sD|\geq 1$ for any state $s$, then $|s_+D|>1$. The same applies for $s_-D$.
\end{proof}

\begin{corol}
Let $D$ be an adequate diagram with $n$ crossings. Then
\begin{enumerate}[label=(\roman*)]
	\item $M_{\kauf{D}}\geq n+2$;
	\item $m_{\kauf{D}}\leq -n-2$.
\end{enumerate}
\end{corol}
\begin{proof}
Since $D$ is adequate, the term with maximal exponent of $\kauf{D}$ derives from $s_+D$ and it does not get cancelled.
\[maximal(\kauf{D})=A^n(-A^{-2}-A^2)^{|s_+D|-1}\spa{5}\stackrel{\mathclap{\mbox{\normalfont\scriptsize\emph{Lemma \ref{lemmaGeq}}}}}{\geq}\spa{5}A^n(-A^{-2}-A^2)^1\implies M_{\kauf{D}}\geq n+2.\]
The analogous can be said about $m_{\kauf{D}}$.
\end{proof}


\section{Link parallels}
This section is mostly designed to extend the results in the previous section to the parallel version of links. Ultimately, we will be able to set a lower bound for the minimal crossing number of parallel links of adequate links. Let us first review what a parallel link is.

Let $D$ be a diagram of an link $L$. We call the \emph{$r$-parallel} of $D$ to the same diagram where each link component has been replaced by $r$ parallel copies of it, all preserving their ``over'' and ``under'' strands as in the original diagram, and write $D^r$. See \emph{Figure \ref{figParall}} for reference.
\begin{figure}[ht]
	\centering
	\includegraphics[scale=0.4]{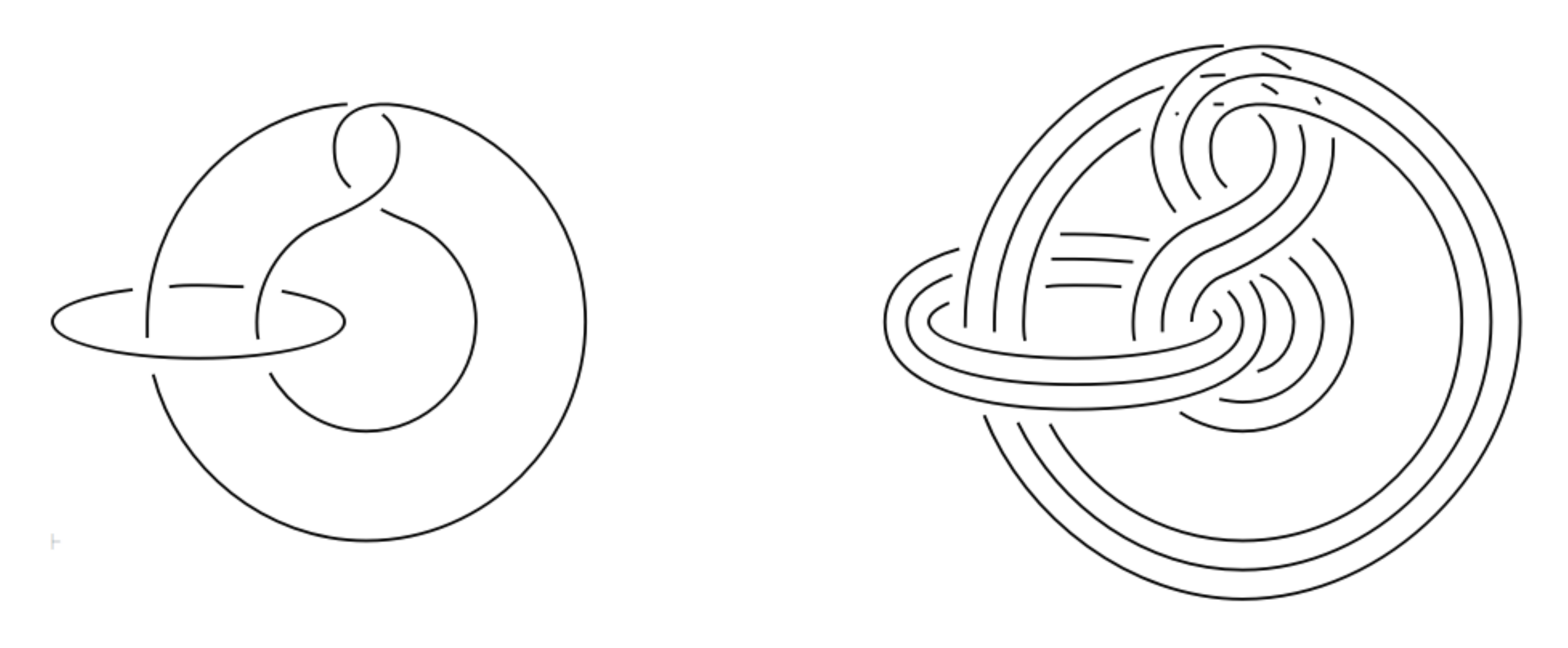}
	\caption{A link diagram and its $3$-parallel.}
	\label{figParall}
\end{figure}

It is important to observe that, by definition, if the original link is oriented, the parallel copies of it will also preserve its orientation. This will play an important roll in the next section of this work.

Sometimes we may refer directly to the $r$-parallel of the link $L$ and write $L^r$, instead of using a diagram of that link. In this case we will be referring to the $r$-parallel of whichever diagram of $L$ we decide to use --- we will only use this notation to express the Jones polynomial of the link, since the result does not depend on the chosen diagram.

Now, we are interested in knowing how properties from the last section change or are preserved when we consider the $r$-parallel of a diagram. Therefore, we first think about the behaviour of the adequacy of a diagram.
\begin{lmm}\label{lemmaParal}
Let $D$ be a link diagram. Then
\begin{enumerate}[label=(\roman*)]
	\item if $D$ is plus-adequate, $D^r$ is also plus-adequate, and
	\item if $D$ is minus-adequate, $D^r$ is also minus-adequate.
\end{enumerate}
\end{lmm}
\begin{proof}
It is enough to observe that $s_+(D^r)=(s_+D)^r$, as it can be extracted from \emph{Figure \ref{figParProof}}.

\begin{figure}[ht]
	\centering{}
	\includegraphics[scale=0.75]{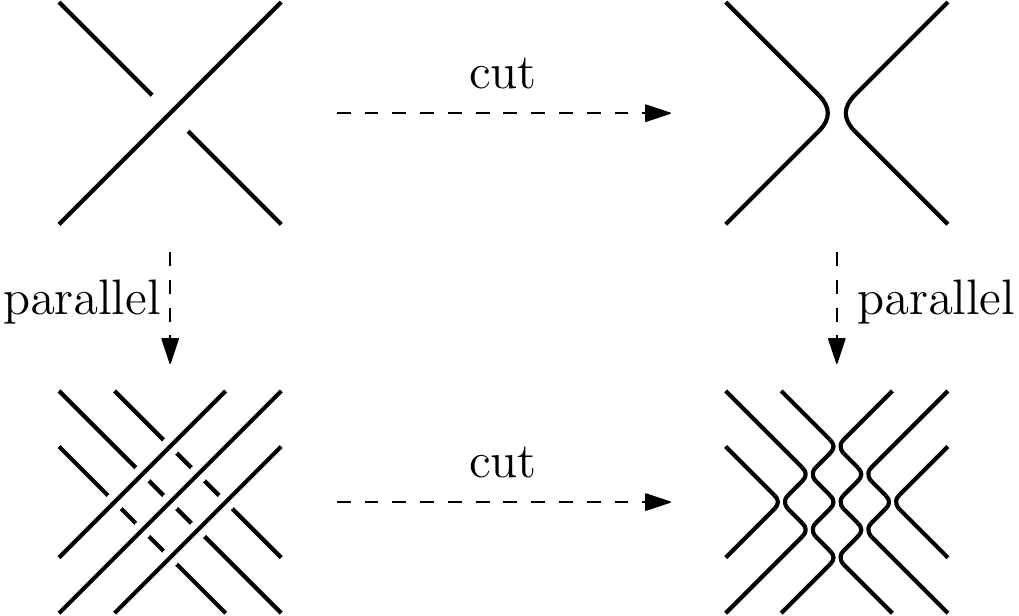}
	\caption{Cut \& parallel = parallel \& cut.}
	\label{figParProof}
\end{figure}

\noindent The same happens with $s_-(D^r)=(s_-D)^r$. This property extended to all crossings gives the formulated result.
\end{proof}

Of course, this implies that adequacy is preserved under the ``taking parallels'' action. In other words, if $D$ is adequate, so $D^r$ is. This is a very important fact, since now it allows us to extend the results in \emph{Section \ref{secAdequacy}} to the $r$-parallel of a diagram.

\begin{lmm}\label{lemmaKaufDr}
Let $D$ be a link diagram with $n$ crossings. Then
\begin{enumerate}[label=(\roman*)]
	\item $M_{\kauf{D^r}}\leq nr^2+2r|s_+D|-2$, with equality if $D$ is plus-adequate, and
	\item $m_{\kauf{D^r}}\geq -nr^2-2r|s_-D|+2$, with equality if $D$ is minus-adequate.
\end{enumerate}
\end{lmm}
\begin{proof}
It suffices to apply the result of \emph{Lemma \ref{lemmaLick4}} to the diagram $D^r$, where, by construction, every crossing of $D$ transforms into $r^2$ crossings.
\begin{figure}[ht]
	\centering{}
	\includegraphics[scale=0.75]{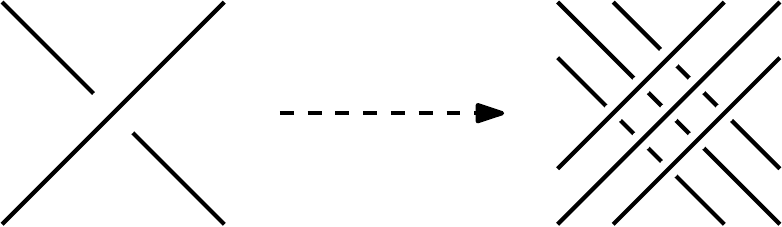}
\end{figure}

\noindent This would affect to the first term of the right-hand side of the inequality, where $n$ crossings for $D$ become $nr^2$ crossings for $D^r$. Next, we use the equalities $s_+(D^r)=(s_+D)^r$ and $s_-(D^r)=(s_-D)^r$ seen in \emph{Lemma \ref{lemmaParal}}, which imply that every circuit generated by the $s_+$ and $s_-$ actions gets parallelized $r$ times, and so $|s_+(D^r)|=r|s_+D|$ and $|s_-(D^r)|=r|s_-D|$. The conditions ``\emph{if $D$ is plus-/minus-adequate}'' for the equality instead of ``\emph{if $D^r$ is plus-/minus-adequate}'' are also consequence of \emph{Lemma \ref{lemmaParal}}, which makes them equivalent.
\end{proof}

\begin{corol}\label{corolKaufDr}
If $D$ is an adequate diagram, then
\[M_{\kauf{D^r}}-m_{\kauf{D^r}}=2nr^2+2r|s_+D|+2r|s_-D|-4.\]
\end{corol}

\begin{exmp}
Let us reflect this result on the usual projection of the trefoil (which is adequate) and its $3$-parallel.
\begin{figure}[H]
	\centering{}
	\includegraphics[scale=0.2]{trefoil}
	\spa{15}
	\includegraphics[scale=0.2]{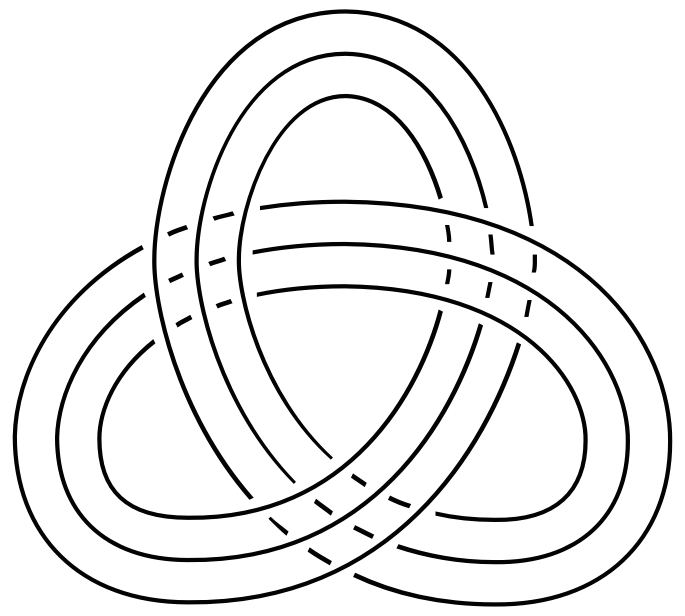}\\
	$\spa{3}3_1$\spa{20}\spa{20}\spa{20}\spa{3}$3_1^3$
\end{figure}
\noindent To begin with, we write down the Kauffman bracket of these knots as reference:
\begin{align*}
	\kauf{3_1}&=A^7+A^3-A^{-5}\textnormal{, and}\\
	\kauf{3_1^3}&=-A^{97}+A^{93}+A^{85}-A^{69}-3A^{61}+2A^{57}-A^{53}+A^{49}-A^{45}+A^{41}+A^{33}+A^{25}+A^{17}.
\end{align*}

\noindent Now, in order to apply \emph{Corollary \ref{corolLick5}} we first draw the $s_-$ and $s_+$ states of $3_1$:
\begin{figure}[H]
	\centering{}
	\includegraphics[scale=0.2]{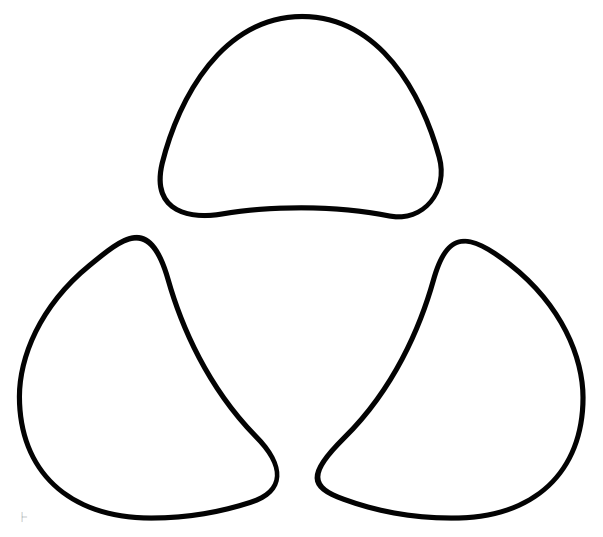}
	\spa{15}
	\includegraphics[scale=0.2]{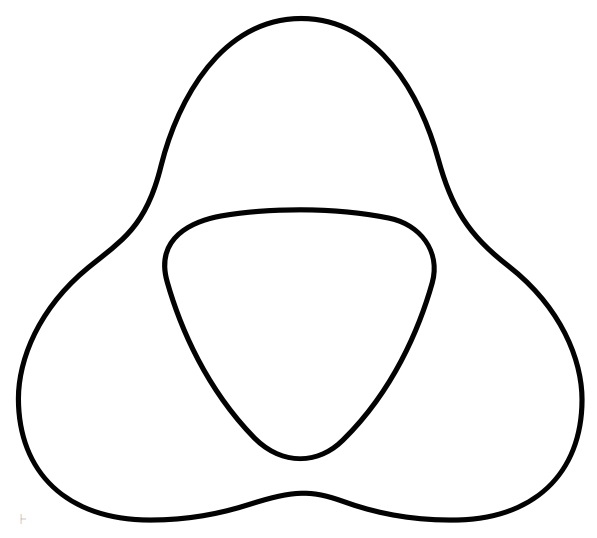}\\
	$\spa{3}s_+(3_1)$\spa{20}\spa{20}\spa{9}$s_-(3_1)$
\end{figure}

\noindent Applying the result of the corollary, we get that
\[B(\kauf{3_1})=2n+2|s_+3_1|+2|s_-3_1|-4=12,\]
\noindent which precisely equals to actual breadth of the Kauffman bracket of $3_1$. Using now the information of the number of circuits of the extreme states of $3_1$ and applying \emph{Corollary \ref{corolKaufDr}} to its $3$-parallel, we repeat the respective calculus with $3_1^3$.
\[B(\kauf{3_1^3})=2\cdot3^2n+2\cdot3|s_+3_1|+2\cdot3|s_-3_1|-4=80,\]
\noindent which, of course, is exactly the breadth of $\kauf{3_1^3}$. Having these results, we also know from \emph{Proposition \ref{propBreadth}} the breadth of the Jones polynomials of these knots:
\begin{align*}
	B(J(3_1))&=\frac{B(\kauf{3_1})}{4}=3\textnormal{, and}\\
	B(J(3_1^3))&=\frac{B(\kauf{3_1^3})}{4}=20.
\end{align*}

\noindent As we have just seen, these results prove to be valuable for calculating the breadth of the Kauffman bracket and Jones polynomial just by knowing some information about one adequate diagram.
\end{exmp}

Our final aim in this section is to set a lower bound for the minimal number of crossings of link parallels when the original knot is adequate. Making use of the results in this section, a lower bound can be set as shown in the following result.

\begin{thm}\label{thmCrossParal}
Let $L$ be an adequate oriented link and $r\in\mathbb{N}$. Then
\[c(L^r)\geq \frac{r^2}{2}c(L)+2r-1.\]
\end{thm}
\begin{proof}
By \emph{Theorem \ref{keythm}} we know that $c(L^r)\geq B(J(L^r))$. Since $L$ is an adequate link, let $D$ be an adequate diagram of $L$ with $n$ crossings. For this particular $D$,
\[c(L^r)\geq B(J(L^r))=\frac{B(\kauf{D^r})}{4}=\frac{2nr^2+2r|s_+D|+2r|s_-D|-4}{4}.\]
This is due to \emph{Corollary \ref{corolKaufDr}}. We now recall \emph{Lemma \ref{lemmaGeq}}, which set a lower bound both for $|s_+D|$ and $|s_-D|$. Using this,
\[c(L^r)\geq \frac{2nr^2+4r+4r-4}{4}=\frac{r^2}{2}n+2r-1\geq \frac{r^2}{2}c(L)+2r-1,\]
this last inequality holding because $c(L)$ is minimal.
\end{proof}

This lower bound is as good as we can prove it for general adequate knots. A tighter bound would require in turn a tighter lower bound for $|s_+D|$ and $|s_-D|$, but for any general adequate diagram of a link it is not possible to say any better than $|s_{\pm}D|\geq 2$.
\begin{corol}\label{corolBeau}
In particular, if $r>1$ these strict inequalities hold:
\begin{enumerate}[label=(\roman*)]
	\item $c(L^r)> \frac{r^2}{2}\cdot c(L)$, and
	\item $c(L^r)> r\cdot c(L)$.
\end{enumerate}
\end{corol}
\begin{proof}
The results are direct from the inequality in \emph{Theorem \ref{thmCrossParal}}.
\end{proof}
This theorem can be further tuned when $L$ is an alternating link.
\begin{corol}\label{corolAlt}
If $L$ is alternating
\[c(L^r)\geq \frac{r(r+1)}{2}c(L)+r-1.\]
\end{corol}
\begin{proof}
Since $D$ is alternating, from \emph{Lemma \ref{lemmaLick6}} we know that $|s_+D|+|s_-D|=n+2$. Using this in the proof \emph{Theorem \ref{thmCrossParal}} gives us the desired result.
\end{proof}
Please note that although one could think that this should be an equality because the equality of \emph{Theorem \ref{keythm}} holds for alternating links, the equality is not assured to be accomplished since $L^r$ is never an alternating link, even if $L$ is.

\section{Satellite knots}

In this last section we prove the main result of this work, which has been already described in the introduction. Nonetheless, this result came after multiple efforts in trying to understand the structure and construction of satellite knots, and the precise calculation of their Kauffman bracket and Jones polynomial. This is why we will first give some results on specific cases of satellites constructions, and prove the main theorem in the latter part of this section.

First, let us define the notation that we will be using. Let $P\subset ST$ (\emph{pattern}) be a knot in the solid torus, and $C\subset \mathbb{S}^3$ (\emph{companion}) be a usual knot. We then call the \emph{satellite knot} of $P$ and $C$ and write $Sat(P,C)$ to the result of faithfully embedding $ST$ (where $P$ lives) into a tubular neighborhood of $C$. Here ``faithfully'' means that a longitude of $ST$ is mapped onto a longitude of $C$.

The construction of satellite knots can be regarded as a result of ``grafting'' the pattern $P$ with a cable of $C$, which in turn can be created by ``grafting'' a parallel of $C$ with a torus link. For precise referral to this term, we define here the \emph{grafting composition} of two knots. This composition is \underline{not} a new concept in knot theory, but rather its branding, for we will be using this construction in several occasions. Let $R$ (\emph{rootstock}) and $S$ (\emph{scion}) be two knots in $ST$ with positive \emph{wrapping number} (minimal number of windings around the center of $ST$). We then put all crossings of each knot inside a \emph{hot zone} \cite{JP} to leave just the neighborhood of a meridian usable and perform a switch of $k$ strands as shown in \emph{Figure \ref{figSwitch}}.
\begin{figure}[ht]
	\centering{}
	\includegraphics[scale=0.8]{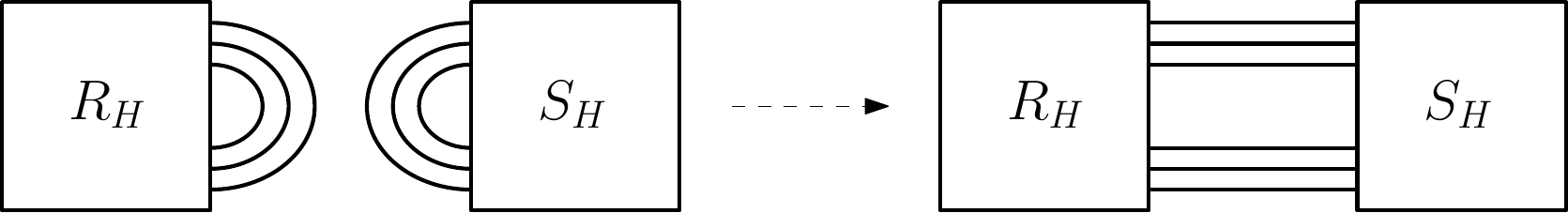}
	\caption{Switch of $3$ strands between hot zones.}
	\label{figSwitch}
\end{figure}

We call the resulting knot the \emph{graft} of $R$ and $S$, and write $R\Cup_k S$ where $k$ is the number of switched strands, which we will call \emph{graft index} or \emph{grafting index}. The next example will better illustrate this process with concrete knots.

\begin{exmp}
Consider the knots $R=3_1^3$ ($3$-parallel of the trefoil) and $S=L(1,2)$ ($(1,2)$-\emph{lasso}), both inside $ST$ with wrapping number $3$, as shown in \emph{Figure \ref{figTorus}}. `` \includegraphics[scale=0.3]{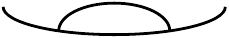} '' stands for the hole of the solid torus.
\begin{figure}[H]
	\centering{}
	\includegraphics[scale=0.2]{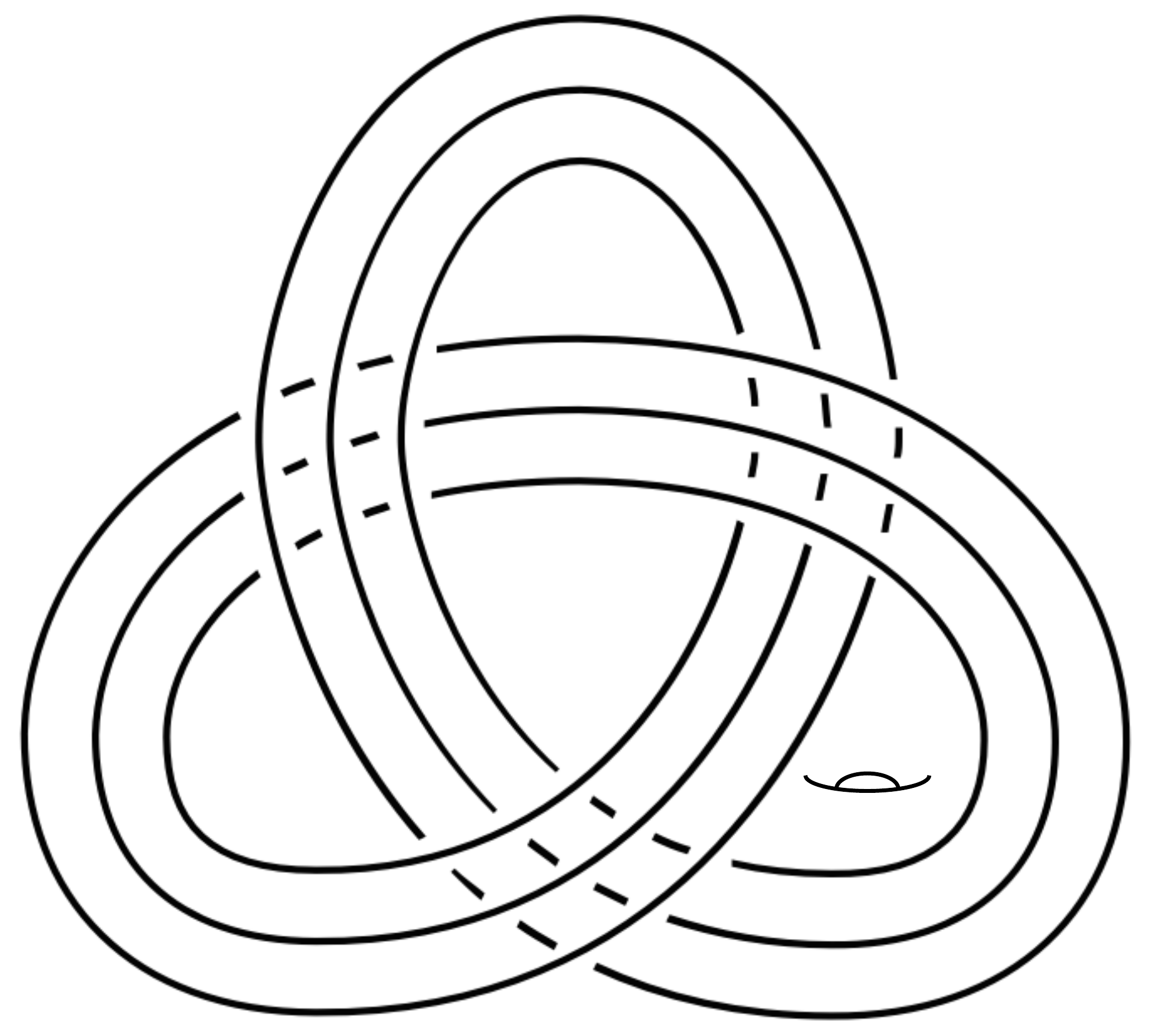}
	\spa{20}
	\includegraphics[scale=0.2]{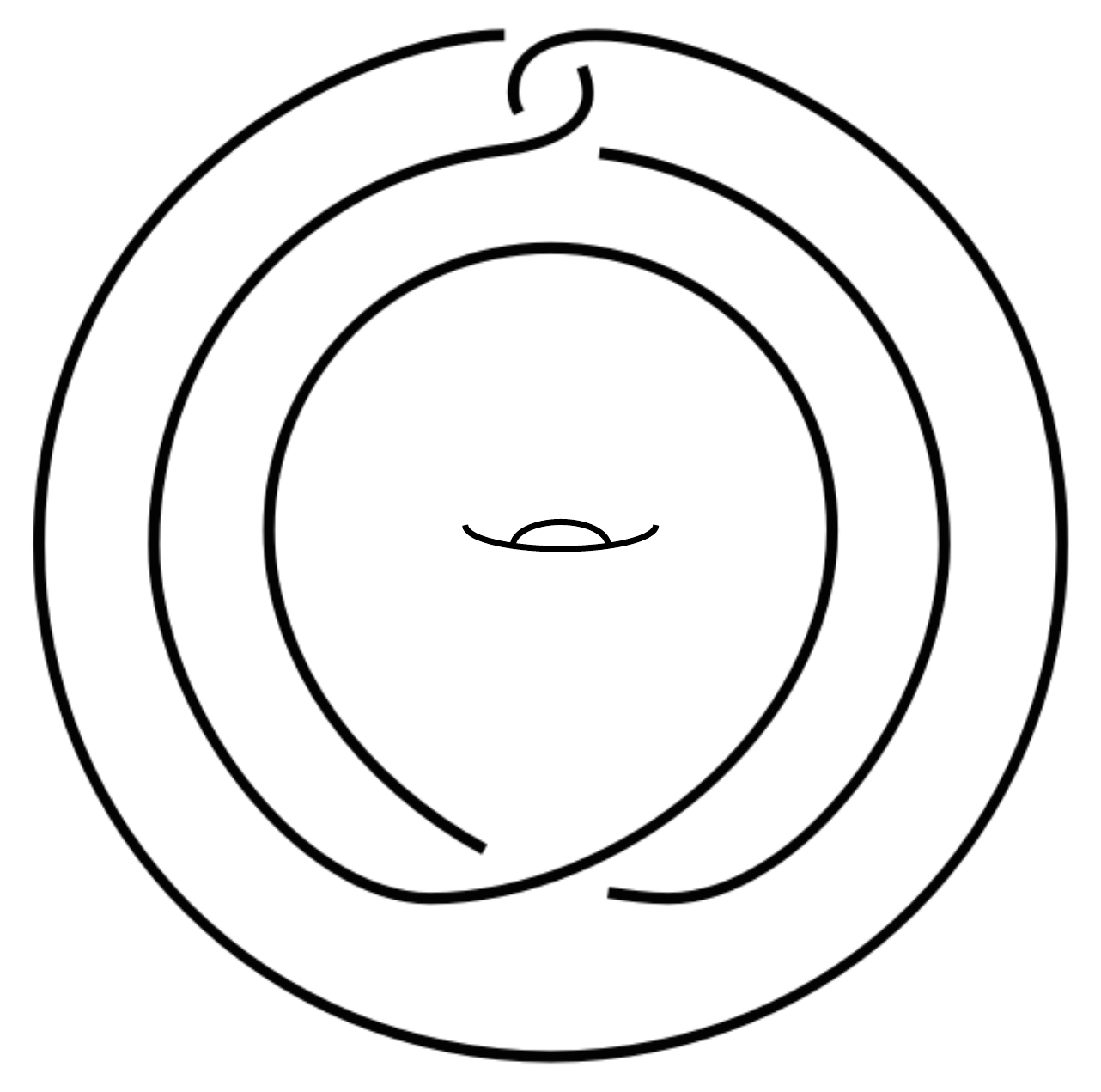}
	\caption{Knots inside $ST$.}
	\label{figTorus}
\end{figure}

\noindent For grafting these knots, we put all their crossings inside hot zones --- $R_H$ and $S_H$ respectively --- leaving outside as many strands as the wrapping number of the knot --- in this case, $3$ each. Then, we perform the strand switch shown in \emph{Figure \ref{figSwitch}} with as many strands as determined by the operation. In our case, we choose $3$ as the grafting index. This whole process is depicted in \emph{Figure \ref{figGraft}}.

\begin{figure}[H]
	\centering{}
	\includegraphics[scale=0.7]{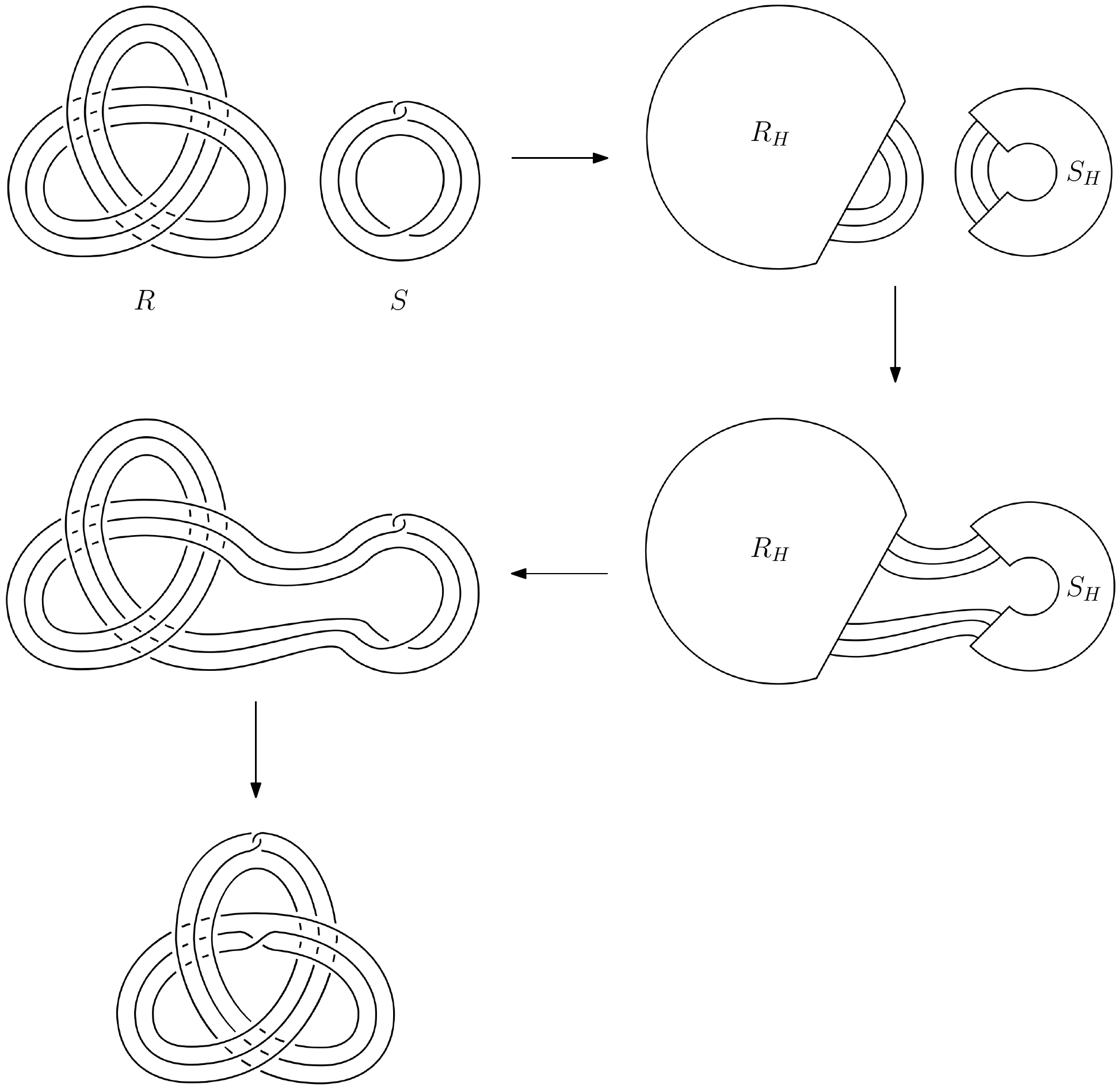}
	\caption{Construction of the graft $3_1^3\Cup_3 L(1,2)$.}
	\label{figGraft}
\end{figure}
\end{exmp}

The result of the grafting composition generally depends on the choice of the meridian where to split and rejoin. Therefore, it is not well defined in this sense --- we need to specify the meridian where to perform the switch. \emph{Figure \ref{figNotWell}} shows two different knots arising from the same original knots $K_1$ and $K_2$ in $ST$ grafting along different meridians --- $m_1$ and $m_2$. As a result, the graft $K_1\Cup_2^{m_1} K_2$ (on the left) has two components whereas the graft $K_1\Cup_2^{m_2} K_2$ (on the right) only has one, proving therefore their inequality.
\begin{figure}[ht]
	\centering{}
	\includegraphics[scale=0.7]{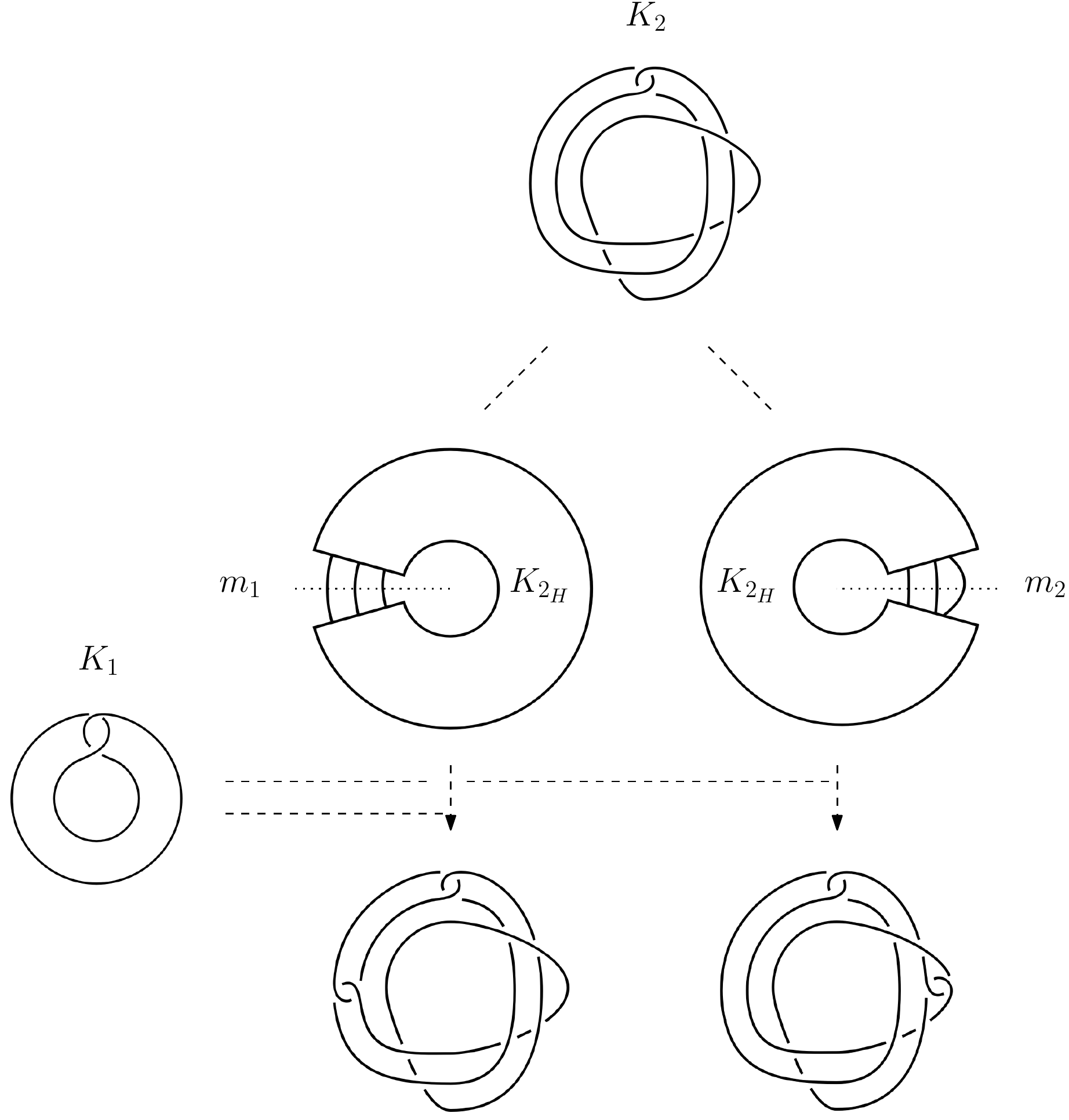}\\
	\qquad\qquad\qquad\qquad $K_1\Cup_2^{m_1} K_2$\spa{20}\spa{8}$K_1\Cup_2^{m_2} K_2$
	\caption{Knot $K_1$ respectively grafted on the left and right side of $K_2$.}
	\label{figNotWell}
\end{figure}
Nonetheless, in our case of study knot parallels will always be one of the two grafting components, therefore in this case the resulting graft will be well defined and we will simply express the grafting operation specifying the number of switched strands --- this is because knot parallels can be traversed inside, allowing the grafting to be executed anywhere without the result changing. 

To better comprehend the general construction of grafts, a more general case of grafting using knot parallels is depicted in \emph{Figure \ref{figGraftExample2}}. The grafting operation is always limited by the minimum of the rootstock's and scion's wrapping number.
\begin{figure}[H]
	\centering{}
	\includegraphics[scale=0.2]{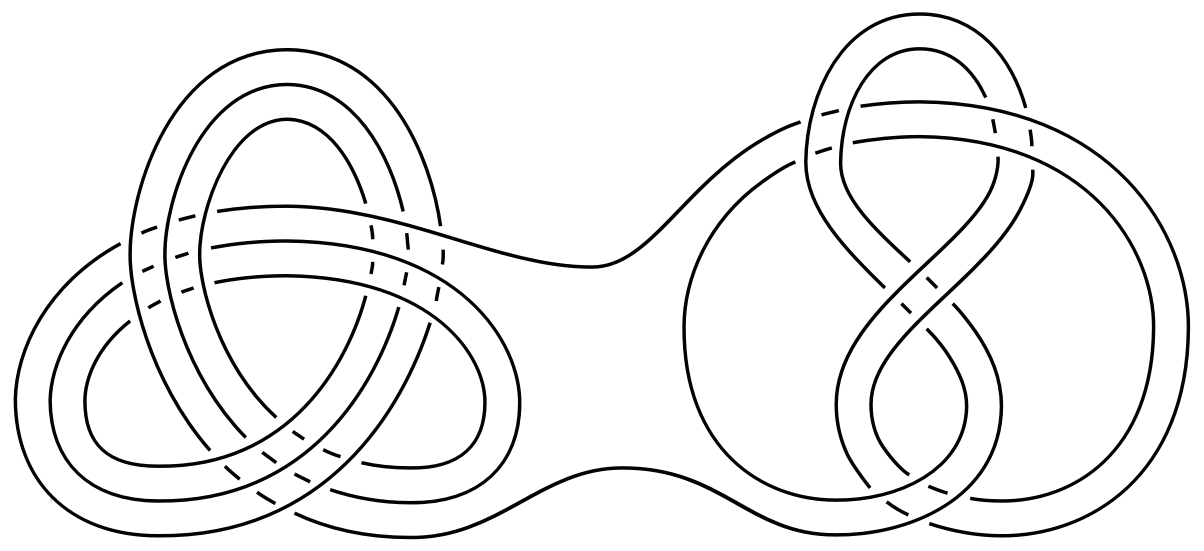}
	\caption{$3_1^3\Cup_14_1^2$.}
	\label{figGraftExample2}
\end{figure}

Please note that in the case of using at least one knot parallel $R\Cup_k S=S\Cup_k R$, that is, the grafting composition is commutative. In particular, if $k=1$ the operation performed is nothing else than the usual composition (\emph{connected sum}) of knots since $K=K^1$ ($1$-parallel), and so grafts can be regarded as an extension of the usual knot sum.
\begin{equation*}
	K_1\#K_2=K_1\Cup_1 K_2.
\end{equation*}

As an example to illustrate how this concept extends the knot sum in general, please regard \emph{Figure \ref{figGraftExample}}, where it is shown how the $r$-parallel of the composition of two knots can be expressed as the $r$-graft of those knots' $r$-parallels. More precisely, $(K_1\#K_2)^k=K_1^k\Cup_kK_2^k$.
\begin{figure}[ht]
	\centering{}
	\includegraphics[scale=0.95]{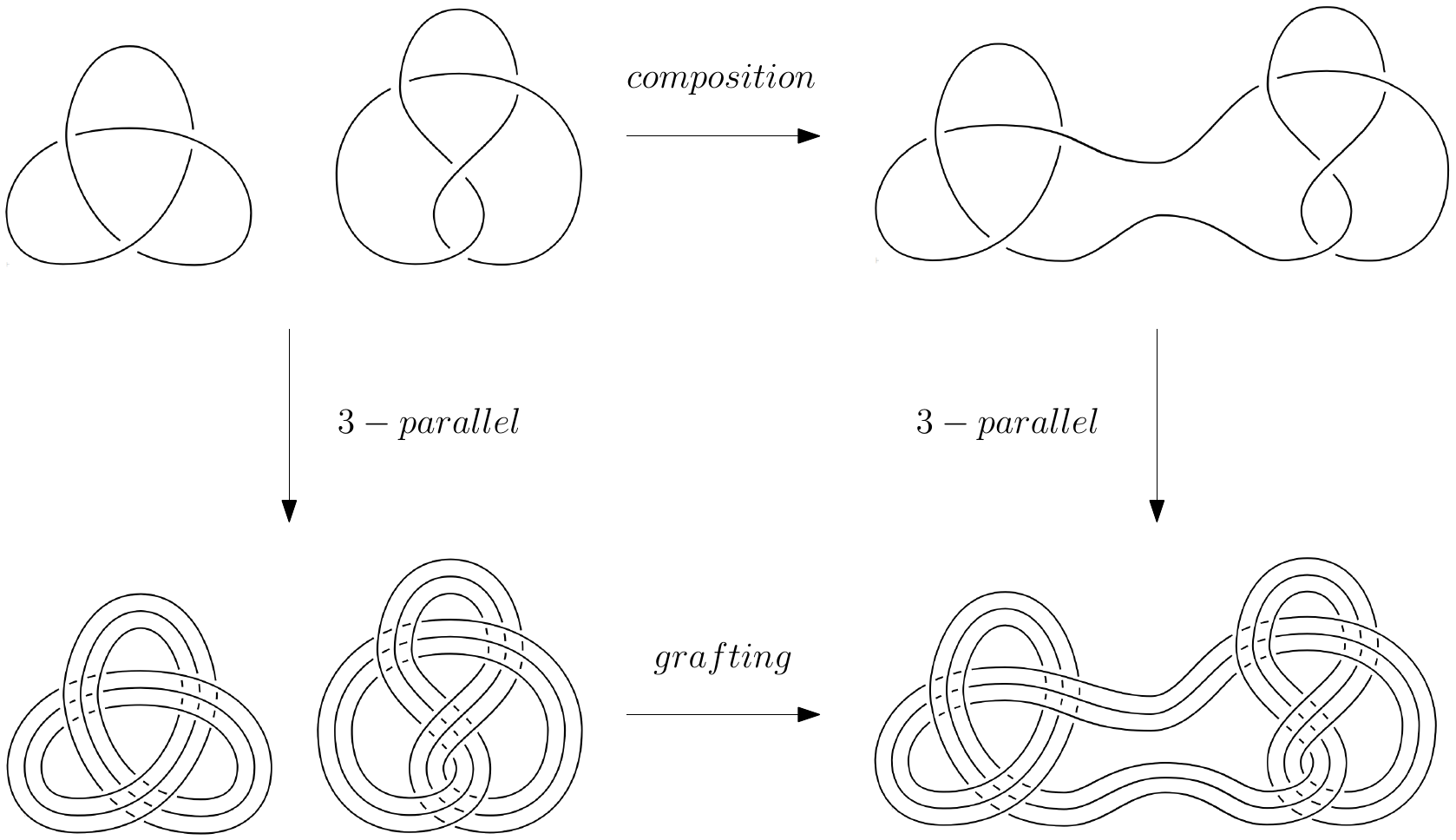}
	\caption{$(3_1\#4_1)^3=3_1^3\Cup_34_1^3$.}
	\label{figGraftExample}
\end{figure}

Again, we ought not to forget that the original knots need to be considered inside the solid torus as in the definition of graft. Nonetheless, since these parallel cases will be our main focus, we will assume that the knot parallels are properly taken inside $ST$ for the graft to work, and simply specify the number of strands that will be switched in the subscript of the grafting operator.

The last definition I want to introduce here is what we call \emph{cable link}. This is also a known concept, yet it is important to recall how it is defined so that we agree on how to handle all these concepts. Let us consider $D$ to be a diagram of $C$, and $wr(D)$ its writhe with blackboard framing. We call the $r$-cable of $D$ and write $(D;r)$ to the $r$-parallel of $D$ with $0$-framing. Conceptually, this is the same as adding $-wr(D)$ twists (\emph{Reidemeister move 1}) to $D$ with blackboard framing and taking its $r$-parallel. In other interpretation, it is the same as taking the $r$-parallel of $D$ with blackboard framing, and then adding $-wr(D)$ full twists to it. This ``adding full twists action'' can be expressed in terms of the newly introduced concept of graft is taking the $r$-parallel of $D$ as rootstock, the torus link $T(-wr(D)r,r)$ as scion and grafting them through all the $r$ strands ($r$-grafting). Both interpretations can be related and compared in \emph{Figure \ref{figMain}}.
\begin{figure}
	\centering{}
	\includegraphics[scale=1]{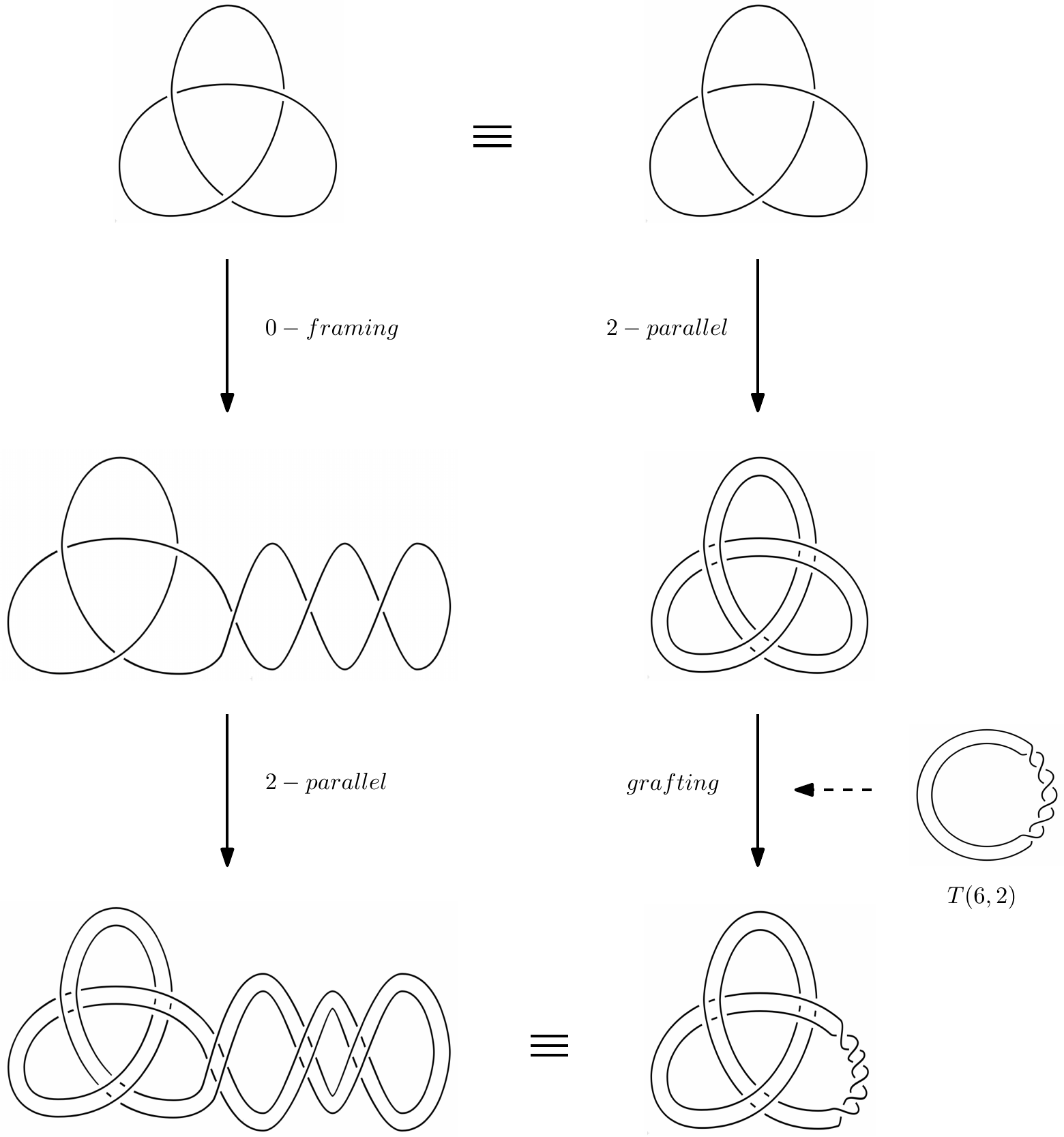}
	\caption{Ways of constructing $(3_1;2)$, the $2$-cable of the trefoil.}
	\label{figMain}
\end{figure}

Bearing all these definitions in mind, we will now present some results regarding grafts.

\begin{thm}\label{thmGraft}
Let $R$ be a knot in $ST$ with wrapping number $w$, and let $S^w$ be the $w$-parallel of a knot $S\subset\mathbb{S}^3$ properly embedded in $ST$ to have wrapping number $w$ as well. Let $R\Cup_wS^w$ be their $w$-graft. Then,
\[J(R\Cup_wS^w)=J\subST(R)\Big|_{z^k\subST=J(S^k)},\]
where $J\subST(R)$ is the Jones polynomial of $R$ in $ST$ \cite{JP}, and $0\leq k\leq w$.
Consequently, assuming $J\subST(R)=\sum_{k=0}^w\beta_kz\subST^k$ (with $\beta_k\in\mathbb{Z}[t^{-1/2}, t^{1/2}]$ and $\beta_w\neq 0$), we can express the previous equation as
\[J(R\Cup_wS^w)=\sum_{k=0}^w\beta_kJ(S^k).\]
\end{thm}
\begin{proof} 
The proof is straightforward from the diagram depiction of the graft. Starting with the crossings of $R$, we only need to notice that solving $J(R\Cup_wS^w)$ in $\mathbb{S}^3$ (through the Jones polynomial skein relations) is identical to solving $J\subST(R)$ in $ST$. In $ST$, after splitting every crossing of $R$ we are left with the $z\subST^k$ generators of $J\subST(R)$, whereas in $\mathbb{S}^3$ the remains of the crossing splits of the $R$ part of $R\Cup_wS^w$ are the same where each $z\subST^k$ (circle around the center of the torus) has been replaced by $k$ parallel copies of $S$.
\end{proof}

As previously noted, when both knots have wrapping number $1$, the graft of $R$ and $S$ is their knot sum. This is reflected in the following corollary.
\begin{corol}
Let $R$ and $S$ be two knots in $\mathbb{S}^3$, and let $R\#S$ be their connected sum.
\[J(R\#S)=J(R)\cdot J(S).\]
\end{corol}
\begin{proof}
This is a well known result which would not need to be proven here. However, it can be newly proven as consequence of \emph{Theorem \ref{thmGraft}} using no previous knowledge of other proofs, which we think adds value to this alternative proof.\\
We just need to recall the observation $J(R\Cup_1 S)=J(R\#S)$ when $R$ and $S$ are considered in $ST$ with wrapping number $1$, and note they can be regarded as $R^1$ and $S^1$ respectively ($1$-parallels). Then, since the wrapping number of $R$ is $1$, we know that its Jones polynomial in $ST$ is equal to its usual Jones polynomial times the $z\subST^1$ generator: $J\subST(R)=J(R)z\subST^1$. Therefore,
\[J(R\#S)=J(R\Cup_1 S)=J\subST(R)\Big|_{z^k\subST=J(S^k)}=J(R)z\subST^1\Big|_{z^1\subST=J(S^1)}=J(R)\cdot J(S),\]
proving the expected result in an alternative way.
\end{proof}

With all this background we now proceed to present and proof this cornerstone lemma. The notation remains the same as in the previous pages.

\begin{lmm}\label{lemmamain}
Let $K\subset \mathbb{S}^3$ be an oriented adequate knot and $r\in \mathbb{N}$. Then:
\[B(J(K;r))\geq B(J(K^r)).\]
\end{lmm}

\begin{proof}
If $r=1$ the statement is trivial since $(K;1)=K^1=K$. Let us assume that $r\geq 2$.
We start by analyzing how $(D;r)$ is built. For this proof, we will, of course, choose $D$ to be an adequate diagram of $K$. Please notice that adequate implies ``reduced'', which means that there are no crossings such as the ones in \emph{Figure \ref{figReduced}}, where $D_1$ are $D_2$ are knot diagrams that connect to the center part. If the diagram was not reduced, then it could be plus- or minus-adequate, but not both at the same time.
\begin{figure}[ht]
	\centering{}
	\includegraphics[scale=0.5]{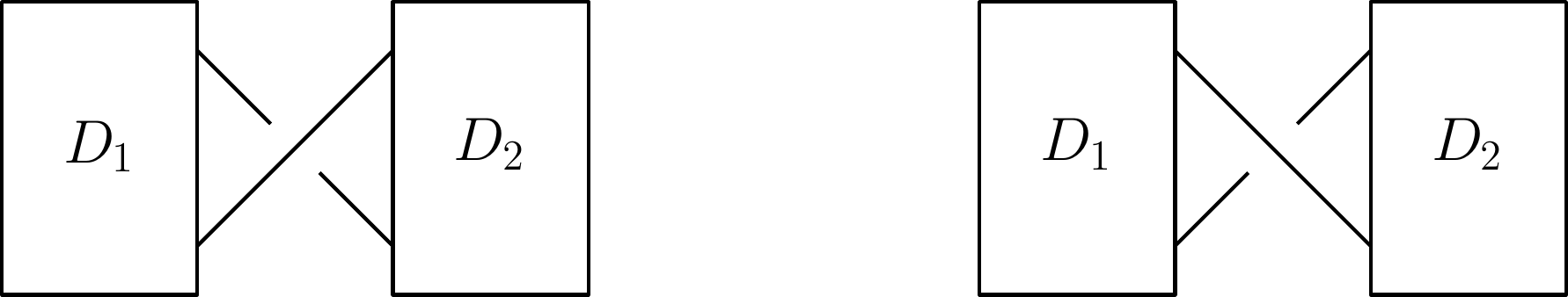}
	\caption{Crossings in non-reduced knots.}
	\label{figReduced}
\end{figure}\\
As we saw in \emph{Figure \ref{figMain}}, taking the $r$-parallel of a $0$-framed diagram $D$ of $K$ is the same as grafting the twisted section of the torus link $T(-w(D)r, r)$ with the $r$-parallel of $D$.

It is important to notice that since the diagram $D$ that we will be using is reduced, adding twists to its parallel will always be a ``crossing-creating'' action --- this is, a created twist will never cancel out with a previously existing reverted twist, since twists in the parallel of a knot only arise from \emph{Reidemeister move 1}, which would cause the diagram not to be reduced. \emph{Figure \ref{figReid}} shows how a twist generates from a non-reduced diagram.
\begin{figure}[ht]
	\centering{}
	\includegraphics[scale=0.5]{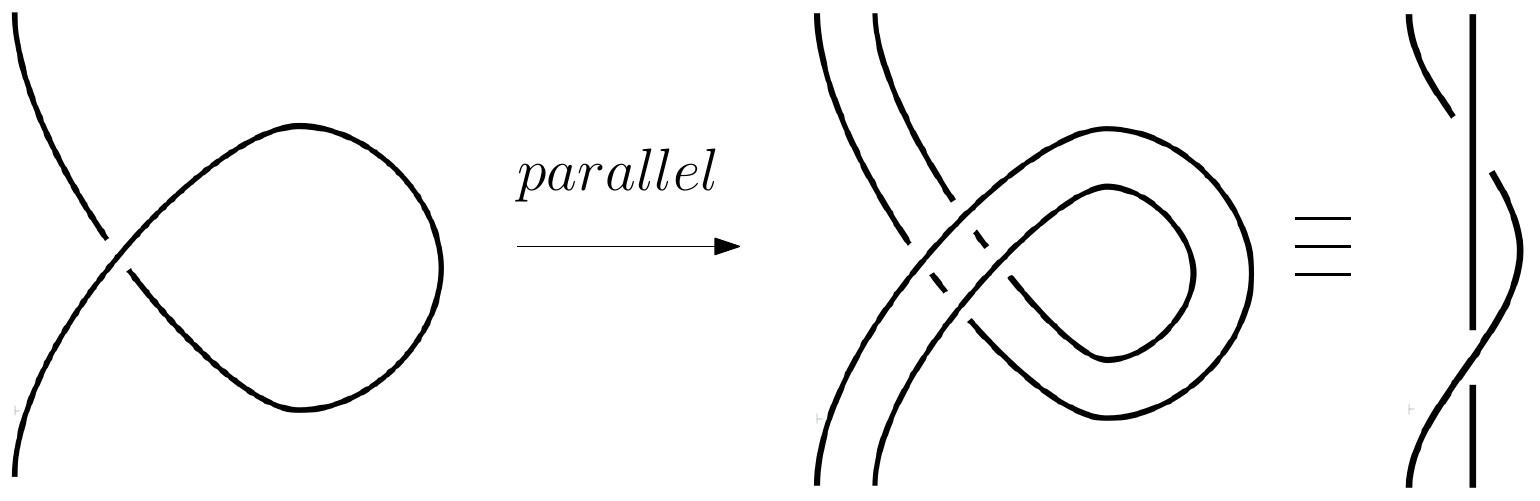}
	\caption{A twist under Reidemeister move 1 becomes a full twist in the parallel.}
	\label{figReid}
\end{figure}\\
In order to prove the formulated inequality, we start by analyzing the Kauffman bracket of $(D;r)$ in terms of the Kauffman bracket of the torus link $T(-w(D)r,r)$, which we will assume that is expressed in $ST$ as
\begin{equation}\label{e0}
\kaufST{T(-w(D)r,r)}=\sum_{k=0}^r\alpha_kz\subST^k, \spa{5}\alpha_r\neq 0,
\end{equation}
with $\alpha_k\in \mathbb{Z}[A^{-1},A]$. Following the right side stream in \emph{Figure \ref{figMain}} and applying the Kauffman bracket version of \emph{Theorem \ref{thmGraft}}, the Kauffman bracket of $(D;r)$ can be then calculated as:
\begin{equation}\label{e1}
\kauf{D;r}=\kaufST{T(-w(D)r,r)}\Big|_{z^k\subST=\kauf{D^k}}=\sum_{k=0}^r\alpha_k\kauf{D^k}, \spa{5}\alpha_r\neq 0.
\end{equation}
In particular, we can estimate the breadth of $\kauf{D;r}$ using the maximum and minimum exponents of $A$ in the last addend of \emph{equation \ref{e1}}, since for any maxima and minima their difference will always be greater or equal than any difference of the involved terms:
\begin{equation*}
\max(a_1,a_2,...,a_n)-\min(b_1,b_2,...,b_m)\geq a_i-b_j, \spa{5}\forall i\in\{1,2,...,n\}, j\in\{1,2,...,m\}.
\end{equation*}
In our case:
\begin{equation}\label{emax}
B(\kauf{D;r})=M_{\kauf{D;r}}-m_{\kauf{D;r}}\geq M_{\alpha_r\kauf{D^r}}-m_{\alpha_r\kauf{D^r}}.
\end{equation}
Since we know $\alpha_r\neq 0$, we can separate the terms on the right side as
\begin{align*}
M_{\alpha_r\kauf{D^r}}&-m_{\alpha_r\kauf{D^r}}=(M_{\alpha_r}+M_{\kauf{D^r}})-(m_{\alpha_r}+m_{\kauf{D^r}})\\
&=(M_{\alpha_r}-m_{\alpha_r})+(M_{\kauf{D^r}}-m_{\kauf{D^r}})=B(\alpha_r)+B(\kauf{D^r}).
\end{align*}
Finally, we only need to notice that $B(\alpha_r)$ is always $0$. This is because there is only one state for any torus link $\kaufST{T(-w(D)r,r)}$ that generates $z\subST^r$. This depends on the writhe of $D$, but whichever the case, $T(-w(D)r,r)$ will be either plus- or minus-adequate. \emph{Figure \ref{figTorusAdeq}} shows how a torus link generated for a diagram $D$ with $w(D)<0$ is plus-adequate. 
\begin{figure}[H]
	\centering{}
	\includegraphics[scale=0.55]{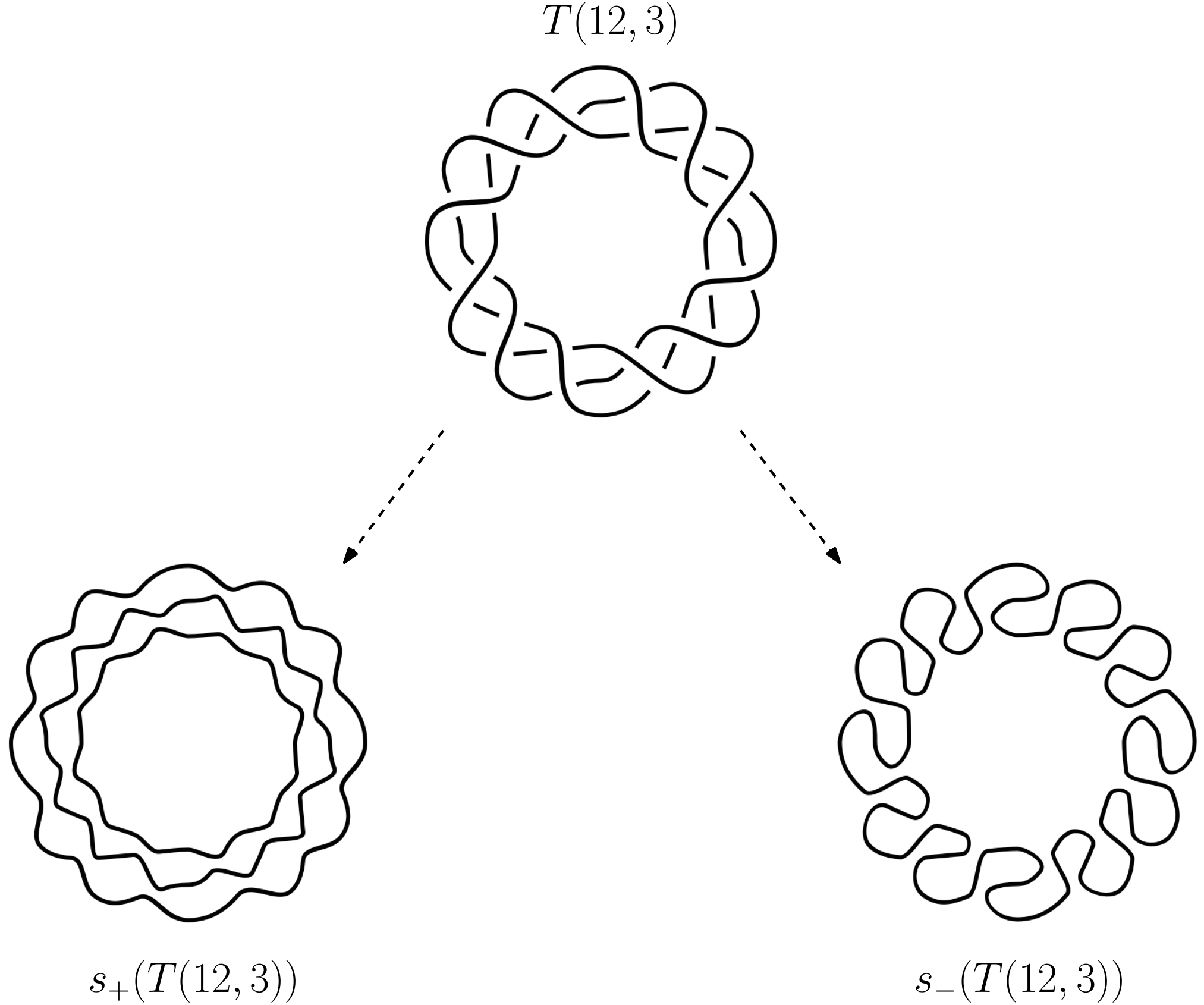}
	\caption{$T(12,3)$ is plus-adequate but not minus-adequate.}
	\label{figTorusAdeq}
\end{figure}
\noindent Let us assume that $w(D)<0$ and let $c(T)$ be the number of crossings of $T(-w(D)r,r)$, for short. Then we know that the torus link is plus-adequate, and the term accompanying $z\subST^r$ is $\alpha_r=A^{c(T)}$ --- all crossings break positively and only circuits around the center of the torus are generated. Therefore $M_{\alpha_r}=m_{\alpha_r}=c(T)$. If $w(D)>0$ the torus link is minus adequate and $\alpha_r=A^{-c(T)}$.
\end{proof}

This lemma becomes a key for proving the main theorem of this work. It will be used together with the next theorem, which states how to calculate the Jones polynomial of satellite knots in terms of their pattern and companion's respective Jones polynomials. But let us first reintroduce the satellite knots. With all these grafting concepts, satellite knots can be redefined as grafting a pattern $P\in ST$ (with wrapping number $M$) with the result of grafting a companion $C\in \mathbb{S}^3$ with the torus link $T(-wr(C)M,M)$. In other words:
\[Sat(P,C)=P\Cup_M(T(-wr(C)M,M)\Cup_MC^M).\]
This construction can be easily recognized in \emph{Figure \ref{figSat}}.
\begin{figure}[ht]
	\centering{}
	\includegraphics[scale=1]{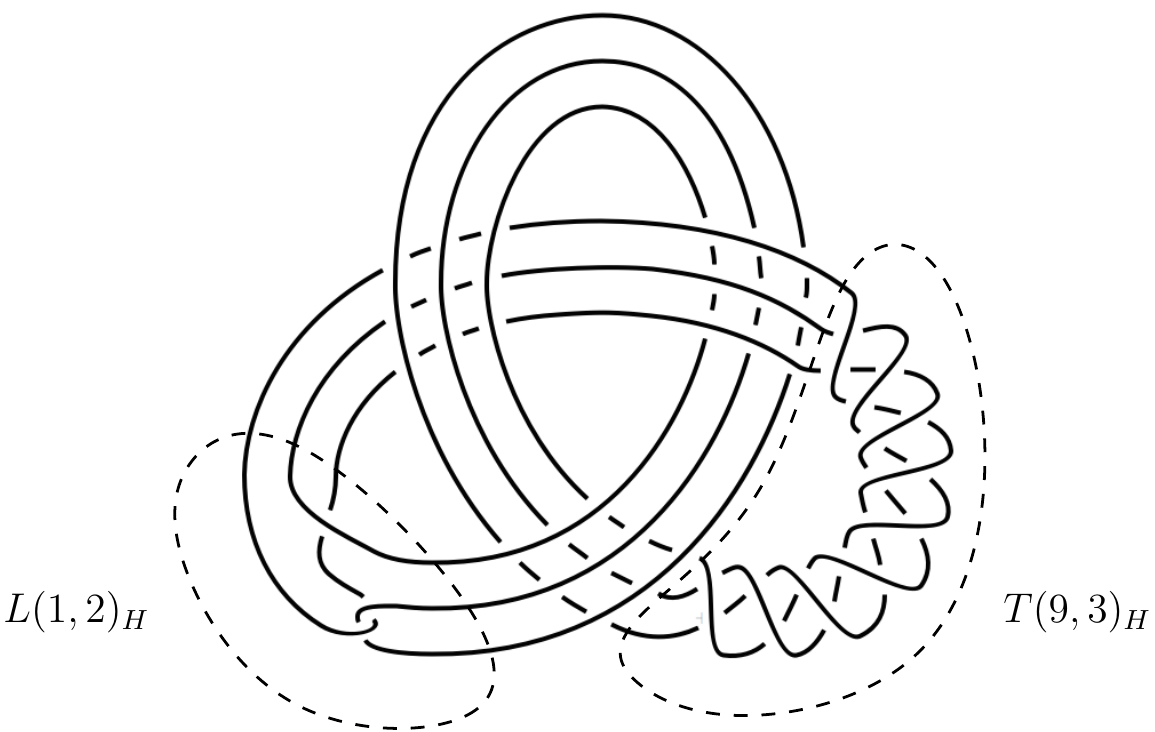}
	\caption{$Sat(L(1,2), 3_1)=L(1,2)\Cup_3(T(9,3)\Cup_33_1^3)$.}
	\label{figSat}
\end{figure}

\begin{thm}\label{thmJP}
Let $P\subset ST$ such that $J\subST(P)=\sum_{k=0}^M\beta_k z^k\subST$ with $\beta_M\neq 0$, let $C\subset \mathbb{S}^3$, and let $Sat(P,C)$ be their satellite knot. Then,
\[J(Sat(P,C))=\sum_{k=0}^M\beta_k J(C;k),\]
where $J(C;k)$ is the Jones polynomial of the $k$-cable of $C$.
\end{thm}
\begin{proof}
The proof to this result can be found in \cite{JP} (\emph{Theorem 3}).
\end{proof}

With these results in mind, we now proceed to present and prove our main result.

\begin{thm}\label{thmMain}
Let $P\subset ST$ such that $J\subST(P)=\sum_{k=0}^M\beta_k z^k\subST$ with $\beta_M\neq 0$, let $C\subset \mathbb{S}^3$ be an adequate knot, and let $Sat(P,C)$ be their satellite knot. Then,
\[c(Sat(P,C))\geq (M_{\beta_M}-m_{\beta_M})+\frac{M^2}{2}c(C)+2M-1.\]
\end{thm}
\begin{proof}
We start by using \emph{Theorem \ref{keythm}} as we previously did, which allows us to set a first lower bound for the crossing number of $Sat(P,C)$.
\begin{equation*}
c(Sat(P,C))\geq B(J(Sat(P,C))).
\end{equation*}
Knowing that $\beta_M\neq 0$ and using \emph{Theorem \ref{thmJP}} and \emph{equation \ref{emax}} we can say that, in particular,
\begin{align*}
B(J(Sat(P,C)))&\geq B(\beta_MJ(C;M))=M_{\beta_MJ(C;M)}-m_{\beta_MJ(C;M)}\\
&=(M_{\beta_M}+M_{J(C;M)})-(m_{\beta_M}+m_{J(C;M)})\\
&=(M_{\beta_M}-m_{\beta_M})+(M_{J(C;M)}-m_{J(C;M)})=B(\beta_M)+B(J(C;M)).
\end{align*}
We now make use of \emph{Lemma \ref{lemmamain}}.
\begin{equation*}
B(\beta_M)+B(J(C;M))\geq B(\beta_M)+B(J(C^M)).{}
\end{equation*}
Finally, as seen in the proof of \emph{Theorem \ref{thmCrossParal}}, this can be explicitly bounded when $C$ is adequate, what leaves us with the coveted and final result:
\begin{equation*}
c(Sat(P,C))\geq (M_{\beta_M}-m_{\beta_M})+\frac{M^2}{2}c(C)+2M-1.
\end{equation*}
\end{proof}
\begin{corol}\label{corolLast}
In particular, if $M>1$ these strict inequalities hold:
\begin{enumerate}[label=(\roman*)]
	\item $c(Sat(P,C))> \frac{M^2}{2}\cdot c(C)$,
	\item $c(Sat(P,C))> M\cdot c(C)$, and
	\item $c(Sat(P,C))> c(C)$.
\end{enumerate}
Being this third inequality the sought answer to Kirby's Problem 1.67.
\end{corol}
\begin{proof}
The proofs are direct, as in \emph{Corollary \ref{corolBeau}}.
\end{proof}

And again, as in the previous section, the proposition can be further tuned for the specific case in which $C$ is an alternating knot.
\begin{corol}
If $C$ is alternating
\[c(Sat(P,C))\geq (M_{\beta_M}-m_{\beta_M})+\frac{M(M+1)}{2}c(C)+M-1.\]
\end{corol}
\begin{proof}
Reflection of \emph{Corollary \ref{corolAlt}}.
\end{proof}

\section{Conclusion}
The main result proven in this work (\emph{Theorem \ref{thmMain}}) gives a partial proof (for adequate knots) to a long-held problem in knot theory, which Kirby expressed in \cite{Ki} (\emph{Problem 1.67}) as ``\textit{Is the crossing number of a satellite knot bigger than that of its companion?},'' and, in his own words, ``Surely the answer is yes, so the problem indicates the difficulties of proving statements about the crossing number.''

It would be of great interest as well to know whether the results and proofs presented in this work could be further extended or tuned to prove this famous problem in general --- without the adequacy condition --- and, more concretely, whether the lower bound could be raised to the expected inequality $c(Sat(P,C))\geq M^2\cdot c(C)$, as argued in \cite{Ho}. 

\section{Acknowledgements}
This work is the epitome of five years of postgraduate research at the University of Tokyo under the direction of Professor Toshitake Kohno, to whom I am profoundly grateful. I would also like to thank Y. Nozaki and Professor T. Kitayama for their helpful discussions and advice. Finally, special thanks to the Japanese Ministry of Education, Culture, Sports, Science and Technology (MEXT) for their financial support throughout these years.


\noindent Graduate School of Mathematical Sciences, \textsc{The University of Tokyo}\\
\emph{E-mail address}: \verb|adri@ms.u-tokyo.ac.jp|
\end{document}